\pdfoutput=1
\documentclass[11pt,a4paper]{article}

\usepackage[utf8]{inputenc}
\usepackage[T1]{fontenc}
\usepackage[english]{babel}

\usepackage{beton}

\usepackage{amsmath,amsthm,amsfonts,amssymb}
\usepackage{graphicx,placeins}
\usepackage[caption=false]{subfig}

\usepackage{algorithm}
\usepackage[noend]{algpseudocode}
\usepackage{xcolor}

\def\myfig#1#2#3{%
    \subfloat[#2 \label{subfig:#3}]{%
        \includegraphics[width=0.45\textwidth]{#1}%
    }%
}

\def\eps{\varepsilon}

\def\abs#1{{\left\lvert #1\right\rvert}}
\def\norm#1{{\left\lVert #1\right\rVert}}

\def\Qbar{{\overline{Q}}}
\def\Qfrak{{\overline{\mathfrak{Q}}}}
\def\Ufrak{{\overline{\mathfrak{U}}}}

\DeclareMathOperator{\cl}{cl}

\DeclareMathOperator{\proj}{proj}

\DeclareMathOperator{\dom}{dom}
\DeclareMathOperator{\epi}{epi}

\DeclareMathOperator*{\argmin}{arg\,min}

\newtheorem{thm}{Theorem}
\newtheorem{lemma}[thm]{Lemma}

\newtheorem{prop}[thm]{Proposition}
\newtheorem{cor}[thm]{Corollary}
\theoremstyle{definition}
\newtheorem{defi}{Definition}

\author{Shabbir \textsc{Ahmed}\\
	Filipe Goulart \textsc{Cabral}\\
	Bernardo \textsc{Freitas Paulo da Costa}}
\date{\today}
\title{Stochastic Lipschitz Dynamic Programming}

\begin{document}
\maketitle

\begin{abstract}
We propose a new algorithm for solving
multistage stochastic mixed integer linear programming (MILP) problems
with complete continuous recourse.
In a similar way to cutting plane methods,
we construct nonlinear Lipschitz cuts
to build lower approximations for the non-convex cost-to-go functions.
An example of such a class of cuts
are those derived using Augmented Lagrangian Duality for MILPs.
The family of Lipschitz cuts we use is MILP representable,
so that the introduction of these cuts does not change the class of the original
stochastic optimization problem.

We illustrate the application of this algorithm on two simple case studies,
comparing our approach with
the convex relaxation of the problems, for which we can apply SDDP,
and for a discretized approximation, applying SDDiP.
\end{abstract}

\section{Introduction}

Non-convex stochastic programming problems arise naturally in models
that consider binary or integer variables, since such variables allow 
the representation of a wide variety of constraints at the cost of inducing 
non-convex feasible sets.
Recent advances in commercial solvers have made possible and more robust 
the solution of several mixed integer linear programming problems, 
broadening the interest of the academic community in studying and 
developing algorithms for the mixed integer stochastic programming area.
Applications such as unit commitment~\cite{Tahanan2015}, \cite{Costley2017},
\cite{knueven2018mixed}, \cite{Ackooij2018}, optimal investment 
decisions~\cite{singh2009dantzig}, \cite{ConejoBook2016}, 
and power system operational planning~\cite{cerisola2012stochastic}, 
\cite{thome2013non}, \cite{Hjelmeland2019}
have driven the development 
of new algorithms for problems that do not fit into the convex 
optimization framework.

Before the development of the MIDAS~\cite{philpott2016midas}
and the SDDiP~\cite{SDDiP2016} algorithms,
the solution
of large-scale multistage stochastic programming problems with  
theoretical guarantees was restricted to convex problems, 
using algorithms such as Nested Cutting Plane~\cite{glassey1973},
Progressive Hedging~\cite{rockafellar1991} and SDDP~\cite{pereira1991multi}.
For a recent reference, see \cite{birge2011introduction} and \cite{shapiro2014lectures}. 
Most of those algorithms build convex underapproximations of the cost-to-go function
at each node of the scenario tree, or at each stage in the stagewise independent setting, to solve the stochastic convex program.
Even the SDDiP algorithm relies on convex underapproximations,
which are shown to be sufficient for convergence thanks to the
binary discretization of the state variables and the tightness property of
the Lagrangian cuts at binary states.
The MIDAS algorithms is based on a different idea, using
step functions to approximate monotone non-convex cost-to-go functions.

The aim of this paper is to propose a new algorithm
for mixed integer multistage stochastic programming problems,
which does not discretize the state variable
and uses non-linear cuts
to capture non-convexities in the future cost function.
Inspired by the exactness results from~\cite{ahmed2017ald},
we build non-convex underapproximations of the expected cost-to-go function,
whose basic pieces are \emph{augmented Lagrangian cuts}.
They can be calculated from augmented Lagrangian duality,
which~\cite{ahmed2017ald} have shown to be exact for mixed-integer linear problems
when the augmentation function is, for example, the $L^1$-norm.
If the original problem already had pure binary state variables,
then, as it was shown in~\cite{SDDiP2016},
Lagrangian cuts are already tight,
and we can use them in our algorithm by setting
the non-linear term of the augmented Lagrangian to zero.

As we will see, even a countable number of these cuts
cannot describe exactly the value function of a mixed-integer linear program,
even when that function is continuous.
However, the $L^1$ cuts have sufficient structure for our purposes,
in two very important ways:
first, they are Lipschitz functions, which still yield global estimates from local behaviour,
even if those are weaker than what's typical with convexity.
The Lipschitz estimates will be crucial for our convergence arguments,
and the resulting algorithm is therefore called
Stochastic Lipschitz Dynamic Programming (SLDP).

Second, it is possible to represent such $L^1$ cuts using binary variables
and a system of linear equalities and inequalities.
Therefore, we do not leave the class of mixed-integer linear problems
when incorporating $L^1$-augmented Lagrangian cuts in each node/stage of the stochastic problem.
Under some hypothesis for continuous recourse of each stage,
it is possible to prove that
the expected cost-to-go functions of stochastic mixed-integer linear problems are Lipschitz,
and therefore our algorithm can be applied directly.


We have organized this paper as follows: 
the next section motivates the SLDP algorithm
with a study of Lipschitz optimal value functions
and a decomposition algorithm for deterministic Lipschitz optimization.
Then, we present the SLDP algorithm for both the full-tree and sampled scenario cases,
and prove their convergence.
Finally, we illustrate our results with a case study.

\section{Lipschitz value functions}\label{sec:global_opt}


Our SLDP algorithm uses Lipschitz cuts to approximate
the cost-to-go function of a stochastic MILP
in a similar fashion to the nested cutting planes algorithm
for stochastic linear programs~\cite{ruszczynski2003stochastic}.
That is, we also compute lower approximations
that iteratively improve the cost-to-go approximation in a 
neighborhood of the optimal solution.

The purpose of this section is to motivate the use of Lipschitz cuts
for nonconvex functions, especially for optimal value functions
of mixed integer problems.
We start with the definition and basic properties of Lipschitz functions
as well as the convergence proof of an algorithm for Lipschitz optimization
that employs special Lipschitz cuts called \emph{reverse norm cuts}.
The idea of reverse norm cuts
also appears in Global Lipschitz Optimization (GLO),
see~\cite{mayne1984outer}, \cite{meewella1988algorithm} and more recently in~\cite{malherbe17a}.
However, our aim is not to develop a new algorithm for GLO
but to explain the reverse norm cut algorithm
and establish some results to extend it
for stochastic multistage MILP programs.

Then, we recall the definition of Augmented Lagrangian duality
and the exactness results in~\cite{ahmed2017ald}
and argue how they can be used to construct
\emph{augmented Lagrangian cuts} in place of the reverse norm cuts.
This provides a unified framework,
generalizing both nonlinear reverse norm cuts
and the linear Lagrangian cutting algorithms
in the continuous setting of~\cite{thome2013non}
or the binary setting of SDDiP~\cite{SDDiP2016}.

Finally, we show how this theory applies to MILP value functions.

\subsection{Lipschitz functions and reverse norm cuts}\label{subsec:lipschitz_cuts}

Let us recall the definition and some results for Lipschitz functions.
We say that $f: \mathbb{R}^d \rightarrow \mathbb{R}$ is a Lipschitz
function with constant $L > 0$ if for all $x,y \in \mathbb{R}^d$ we
have $|f(x) - f(y)| \leq L \| x - y \|$.
Note that the linear function 
$f(x) = a^\top x$ is Lipschitz 
with constant $L = \|a\|_*$,
where $\|\cdot\|_*$ indicates the dual norm.
Let $f_1$ and $f_2$ be Lipschitz functions with constants $L_1$ and $L_2$, respectively.
It can be shown that
\begin{itemize}
	\item the maximum and minimum of $f_1$ and $f_2$ are Lipschitz functions
	with constant $\max\{L_1,L_2\}$; and
	\item the sum of $f_1$ and $f_2$ is a Lipschitz function with constant~$L_1+L_2$.
\end{itemize}


Consider now an example of a non-convex Lipschitz function.
Let $f$ be the piecewise linear function defined on $[0,3]$ by
\[
f(x) = 
\begin{cases}
  x & \text{: $0 \leq x \leq 1$,} \\
  1 & \text{: $1 \leq x \leq 2$,} \\
3-x & \text{: $2 \leq x \leq 3$,}
\end{cases}
\]
see figure~\ref{fig:milp_challenge} for an illustration.
Note that $f$ can be written as the minimum of linear functions,
$f(x) = \min \{x, 1, 3-x\}$,
so~$f$ is a Lipschitz function with constant~$L = 1$.
\begin{figure}[!hbtp]
	\centering
	\includegraphics[width=0.6\textwidth]{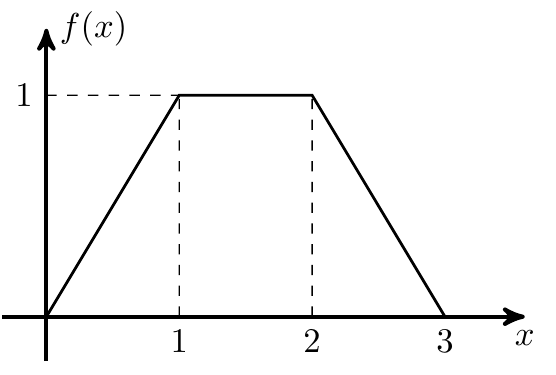}
	\caption{Piecewise linear Lipschitz function.}
	\label{fig:milp_challenge}
\end{figure}
It can also be seen from figure~\ref{fig:milp_challenge}
that the tightest lower convex approximation~$g$ for~$f$ over $[0,3]$
is the zero function.
This implies that there will always be a gap between
linear cuts and the original function~$f$,
since the best they could do is reproduce the Lagrangian relaxation $g$.
In other words, there is no way to use lower linear cuts 
to close the gap with~$f$,
which motivates the introduction of reverse norm cuts.

\begin{defi}
	A \emph{reverse norm cut} for a function $f$,
  centered at $\overline{x}$ and with parameter $\rho$,
	is the function
	\[
	C_{\rho,\overline{x}}(x) = f(\overline{x}) - \rho\cdot\| x - \overline{x}\|.
	\]
\end{defi}
Note that $C_{\rho,\overline{x}}(x)$ is a Lipschitz function with constant~$\rho$,
and if $C_1(x)$ and $C_2(x)$ are reverse norm cuts with parameters~$\rho_1$ and~$\rho_2$
then $\max \{C_1, C_2 \}(x)$ is a Lipschitz function with
constant~$\max\{\rho_1,\rho_2\}$.

Suppose that~$f$ is a Lipschitz function with constant~$L$.
If $\rho$ is greater than or equal to~$L$,
then all functions~$C_{\rho,\overline{x}}(x)$
are valid reverse norm cuts,
for any center point~$\overline{x}$.
Indeed,
\begin{align}\nonumber
C_{\rho,\overline x}(x)
& = f(x) - f(x) + f(\overline x) - \rho\cdot\norm{x - \overline x} \\ \nonumber
& \leq f(x) + L\cdot\norm{x - \overline x}  - \rho\norm{x - \overline x} \\ 
& \leq f(x), \label{eq:tent_ineq}
\end{align}
for all $x \in \mathbb{R}^d$.

A fundamental difference between reverse norm cuts and cutting planes 
is that, in general, one cannot guarantee that a 
piecewise linear Lipschitz function~$f$ can always
be represented as the maximum 
of a finite number of reverse norm cuts.
Indeed, as can be seen in figure~\ref{subfig:milp_cut_1},
one would need all reverse norm cuts centered at every $\overline{x} \in [1,2]$
to represent the non-convex function $f(x)$.
\begin{figure}[!hbtp]
	\centering
	\myfig{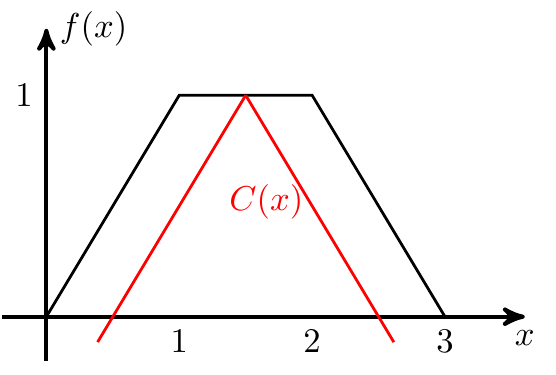}{Reverse norm cut with $\rho = 1$.}{milp_cut_1}
	\hfill
	\myfig{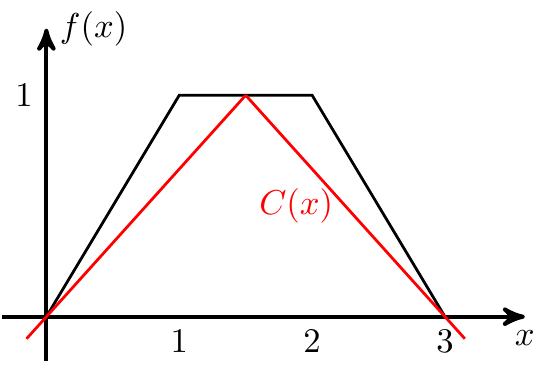}{Reverse norm cut with $\rho = 2/3$.}{milp_cut_2}
	\caption{Example of reverse norm cuts for a Lipschitz function.}
\end{figure}

Sometimes, it is possible
to obtain a reverse norm cut centered at a point~$\overline{x}$ 
with Lipschitz constant $\rho$ less than $L$,
but such reverse norm cut may not be a lower approximation elsewhere
if translated in the domain.
In figure~\ref{subfig:milp_cut_2},
we show a tighter cut with $\rho = 2/3$ centered at $\overline{x} = 1.5$,
but this lower $\rho$ cannot be used for any other point in the domain of $f$.
If the constant $\rho$ is greater than $L$,  
the corresponding reverse norm cut is a lower approximations of~$f$ 
if centered at any point~$\overline{x}$ 
by the same deduction made in~\eqref{eq:tent_ineq}.

Finally, observe that if the function~$f$ is not Lipschitz
then reverse norm cuts may need arbitrarily
large parameters, depending on their center point.
Indeed, let $g(x)$ be the fractional-part function:
\[
  g(x) = x - \lfloor x \rfloor = \min_{y \in \mathbb{Z}, y \leq x} x - y,
\]
which is both piecewise linear and the optimal value function of a MILP,
see figures~\ref{subfig:milp_disc_1} and~\ref{subfig:milp_disc_2}.
\begin{figure}[!hbtp]
	\centering
	\myfig{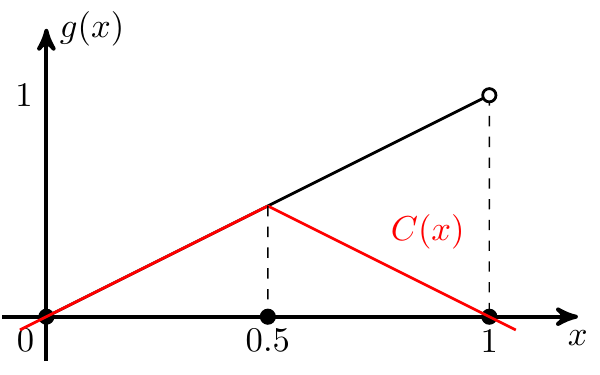}{Reverse norm cut with $\rho = 1$.}{milp_disc_1}
	\hfill
	\myfig{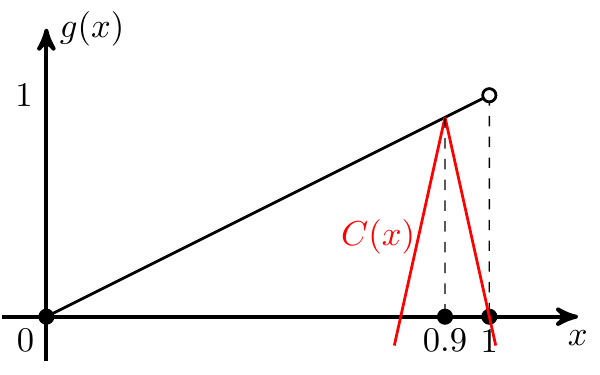}{Reverse norm cut with $\rho = 9$.}{milp_disc_2}
	\caption{Example of reverse norm cuts for a non-Lipschitz function.}
	\label{fig:milp_disc}
\end{figure}%
As the point~$\overline{x}$ approaches~$1$,
the opening of the reverse norm cuts centered at~$\overline{x}$ goes to zero,
so their corresponding Lipschitz constant~$L$
must go to infinity.

\subsection{Optimization with reverse norm cuts}

To develop the intuition for the SLDP method
that will be presented in section~\ref{sec:lddp},
we introduce the reverse cut method for a deterministic problem
and its convergence proof.
We hope that this simpler setting
helps the reader understand some of the basic mechanisms of the proof
without the notational burden of the general stochastic multistage setting.

The deterministic optimization problem that will be investigated is 
\begin{equation}\label{eq:determ_main}
\begin{array}{rl}
\nu= \underset{x}{\min} & f(x) + g(x) \\
\text{s.t.} & x \in X,
\end{array}
\end{equation}
where $f(x)$ is a `simple' function and $g(x)$ is a `complex' one,
so while it is relatively cheap to optimize $f$,
we only assume that it is possible to evaluate~$g$ at given points.
For example, this can be the stage problem in the nested decomposition of a multistage problem,
where $f$ is the immediate cost and $g$ the average cost-to-go.

Assume that $g(x)$ is a Lipschitz function with constant~$L$
and~$X$ is a compact subset of~$\mathbb{R}^d$.
The reverse cut method, presented in algorithm~\ref{alg:cutting_tents},
similarly to decomposition methods,
iteratively approaches~$g(x)$
by reverse norm cuts centered at points
obtained along the iterations of the algorithm.
Those points are solutions of an
approximated optimization problem obtained
from~\eqref{eq:determ_main},
where $g(x)$ is replaced by the current 
Lipschitz approximation~$g^{k}(x)$.

\begin{algorithm}[!htbp]
	\caption{Reverse cut method}\label{alg:cutting_tents}
	\begin{flushleft}
		\textbf{Input:} lower bound~$g^1 \equiv M$,
		                cut constant~$\rho \geq L$, and
		                stopping tolerance~$\eps \geq 0$.\\
		\textbf{Output:} optimal solution~$x^k \in X$ and optimal value~$\nu^k$
						approximations, and Lipschitz lower estimate~$g^k$.
	\end{flushleft}
	Set initial iteration $k=1$ and go to step~1.
	\begin{algorithmic}[1]
		\State Solve the following problem:
		\begin{equation}\label{eq:determ_approx}
		\begin{array}{rl}
		\nu^k = \underset{x}{\min} & f(x) + g^{k}(x) \\
		\text{s.t.} & x \in X,
		\end{array}
		\end{equation}
		and take an optimal solution $x^k \in \argmin_{x \in X} f(x) + g^{k}(x)$.
    \State Evaluate $g(x^k)$. Stop if $g(x^k) - g^{k}(x^k) \leq \eps$.
		Otherwise, go to step 3.
		\State Update the Lipschitz approximation, adding a reverse norm cut at $x^k$:
		\[
		{g}^{k+1}(x) = 
		\max\left\{{g}^{k}(x),\ 
		g(x^k) - \rho \cdot \|x - x^k\| \right\}.
		\] 
		Set $k := k+1$, and go to step 1.
	\end{algorithmic}
\end{algorithm}

The Lipschitz approximation~$g^{k+1}(\cdot)$
is the maximum between the reverse norm cut
centered at~$x^k$ and the previous approximation~$g^{k}(\cdot)$:
\[
{g}^{k+1}(x) = 
\max\left\{{g}^{k}(x),\ 
g(x^k) - \rho \cdot \|x - x^k\| \right\}.
\]
By induction, all $g^k(\cdot)$ are Lipschitz functions
with constant~$\rho$,
since it is the maximum of reverse norm cuts with that constant.

With this observation in mind, we will prove the
convergence of algorithm~\ref{alg:cutting_tents}.
The compactness assumption of~$X$ is used only to ensure
the existence of a cluster point
for the sequence of trial points~$\{x^k\}_{k\in\mathbb{N}}$.

\begin{lemma}\label{lemma:cluster_converg}
	Suppose Algorithm~\ref{alg:cutting_tents} does not meet the 
	stopping criteria for problem~\eqref{eq:determ_main}.
	Let $x^* \in X$ be any cluster point
	of the sequence of trial points~$\{x^k\}_{k\in\mathbb{N}}$
	and let $\mathcal{K}$ be the indices of a subsequence
	that converges to~$x^*$.
	Then $\{g^k(x^k)\}_{k\in\mathcal{K}}$ converges to $g(x^*)$:
	\[
	\lim\limits_{k\in\mathcal{K}} g^k(x^k) = g(x^*).
	\]
\end{lemma}
\begin{proof} 
	We will bound $|g(x^*) - g^k(x^k)|$ by the distance
	between the subsequence $\{x^k\}_{k \in \mathcal{K}}$
	and the cluster point~$x^*$.
	Using the triangular inequality and the Lipschitz definition, we get
	\begin{align}\nonumber
	\left|g(x^*) - g^k(x^k)\right| 
	& \leq  \left| g(x^*) - g(x^k)\right|
	+ \left| g(x^k) - g^k(x^k) \right| \\ \label{eq:cluster_ineq1}
	& \leq L \cdot \| x^* - x^k \| + \abs{g(x^k) - g^k(x^k)}.
	\end{align}

	Since $x^k$ converges to $x^*$ along $k \in \mathcal{K}$,
	we only need to prove that the lower bounds $g^k$
	will get arbitrarily close to $g$ at the trial points,
	before it is updated.
	By construction of the reverse cut method,
	$g^k(x^k)$ is less than or equal to~$g(x^k)$, so we only need upper bounds.
  Let $j$ be the index just before $k$ in the subsequence $\mathcal{K}$.
  Since $g^k$ is $\rho$-Lipschitz,
  \[
    g^k(x^k) \geq g^k(x^j) - \rho\cdot \| x^k - x^j \| .
  \]
  But also by construction, $g^k(x^j) = g(x^j)$,
  so subtracting the equation above from $g(x^k)$
  and applying once again the Lipschitz hypothesis on $g$ we obtain:
	\begin{align}\nonumber
	g(x^k) - g^k(x^k)
	& \leq g(x^k) - g(x^j) + \rho\cdot \| x^k - x^j \| \\ \label{eq:cluster_ineq2}
	& \leq L \cdot \| x^k - x^j \| + \rho \cdot \| x^k - x^j \|.
	\end{align}
	By replacing~\eqref{eq:cluster_ineq2} in~\eqref{eq:cluster_ineq1}
	and taking the limit over~$\mathcal{K}$, we conclude that the 
	sequence~$\{g^k(x^k)\}_{k\in\mathcal{K}}$ converges to~$g(x^*)$.	
\end{proof}

From the convergence result above,
we now separate the analysis of the cases where $\eps$ is strictly positive and zero
respectively in Corollary~\ref{cor:finite_deterministic} and Theorem~\ref{thm:cvg_rnm} below.

\begin{cor}\label{cor:finite_deterministic}
	For any stopping tolerance~$\eps>0$, Algorithm~\ref{alg:cutting_tents}
	stops in a finite number of iterations with an $\eps$-optimal solution
	for~\eqref{eq:determ_main}.
\end{cor}
\begin{proof}
	From inequality~\ref{eq:cluster_ineq2} in the proof of Lemma~\ref{lemma:cluster_converg}
	and compactness of $X$,
	Algorithm~\ref{alg:cutting_tents}
	will stop in a finite number of iterations if the stopping tolerance~$\eps$
	is strictly positive.
	Using the fact that $x^k \in X$ is a feasible solution
	to~\eqref{eq:determ_main} and optimal solution to~\eqref{eq:determ_approx}   such that~$g(x^k)-g^k(x^k) \leq \eps$,
	where that~$g^k$ is a lower estimate for~$g$, we have the following inequalities:
	\[
	\nu^k \leq \nu \leq f(x^k) + g(x^k) \leq f(x^k) + g^k(x^k) + 
	[g(x^k) - g^k(x^k)] \leq \nu^k +\eps,
	\]
	which means that $f(x^k) + g(x^k)$ is bounded between~$\nu$ and~$\nu+\eps$.
\end{proof}

\begin{thm} \label{thm:cvg_rnm}
	Consider the stopping tolerance~$\eps$ equal to zero.
	Then, Algorithm~\ref{alg:cutting_tents} stops with an optimal 
	solution in a finite number of iterations or it generates
	a sequence of optimal value approximations~$\{\nu^k\}_{k\in\mathbb{N}}$ 	
	that converges to the optimal value~$\nu$ of problem~\eqref{eq:determ_main} 
	and every cluster point~$x^*$ of the sequence~$\{x^k\}_{k\in\mathbb{N}}$
	is a minimizer of~\eqref{eq:determ_main}.
\end{thm}
\begin{proof}
	Suppose that Algorithm~\ref{alg:cutting_tents} stops after a finite
	number of iterations. Using the same argument of Corollary~\ref{cor:finite_deterministic},
	we obtain that the last trial point~$x^k$ is the optimal solution 
	to~\eqref{eq:determ_main}.
	
	Now, suppose that Algorithm~\ref{alg:cutting_tents} never reaches
	the stopping condition, so we know from Lemma~\ref{lemma:cluster_converg} 
	that the sequence~$\{g^k(x^k)\}_{k\in\mathcal{K}}$ converges to~$g(x^*)$.
	Moreover, $x^k$ is a feasible solution to
	the main problem~\eqref{eq:determ_main}
	and optimal solution to the approximate
	problem~\eqref{eq:determ_approx}.
  Since $g^k$ is a lower estimate for $g$, we obtain the following relationships:
	\begin{equation}\label{eq:rel_sol}
	f(x^k) + g(x^k) \geq \nu \geq \nu_k = f(x^k) + g^k(x^k).
	\end{equation}
	Taking the limit over~$\mathcal{K}$ on both sides of~\eqref{eq:rel_sol},
  and by continuity of $f$ and $g$, we obtain:
	\begin{equation*}
	f(x^*) + g(x^*) \geq 
	\nu \geq \lim\limits_{k \in \mathcal{K}} \nu_k 
	=  f(x^*) + g(x^*),
	\end{equation*}
  which shows that all inequalities above are equalities,
  and therefore $x^*$ is an optimal solution
	to~\eqref{eq:determ_main}.	

	Going back to the full sequence, we recall that
	the sequence of objective functions  
	$\{f(x) + g^k(x)\}_{k\in\mathbb{N}}$
	is monotone nondecreasing, 
	so the sequence of optimal 
	values~$\{\nu_k\}_{k\in\mathbb{N}}$ is
	also monotone and nondecreasing, 
	which implies that~$\{\nu_k\}_{k\in\mathbb{N}}$ also
	converges to~$\nu$. 
\end{proof}

Observe that the proof of Lemma~\ref{lemma:cluster_converg},
Corollary~\ref{cor:finite_deterministic} and Theorem~\ref{thm:cvg_rnm} 
use only two properties of the reverse-norm cuts.
Besides their Lipschitz character,
they are exact at trial points, that is, $C_{\overline{x}}(\overline{x}) = g(\overline{x})$.
Therefore, any other way of producing \emph{uniformly Lipschitz tight cuts} for $g$
yields a convergent Lipschitz cut method for optimizing $f + g$.

So, just before presenting a class of MILPs with Lipschitz value functions,
we show how augmented Lagrangian duality can be used to produce tight Lipschitz cuts.

\subsection{Augmented Lagrangian cuts}
\label{sec:ald_cuts}

Recall the definition of augmented Lagrangian duality for
mixed-integer optimization problems.
Let
\begin{equation}
  \label{milp}
  \begin{array}{rl}
    g(b) = \min_{x \in X} & c^\top x \\
    \text{s.t.} & Ax = b
  \end{array}
\end{equation}
be a parameterized linear problem,
where $X$ describes the mixed-integer constraints.

\begin{defi}
Given an augmenting function $\psi$,
which is non-negative and satisfies $\psi(0) = 0$,
the \emph{augmented Lagrangian} for problem~\eqref{milp} is given by
\begin{equation}
  \label{al_milp}
   g^{AL}(b; \lambda, \rho) = \min_{x \in X} c^\top x - \lambda^\top(Ax - b) + \rho\cdot\psi(Ax - b).
\end{equation}
\end{defi}

Since any feasible solution to the original optimization problem~\eqref{milp}
remains feasible with the same objective value for the augmented Lagrangian~\eqref{al_milp},
we see that, for any $b$, $\lambda$ and $\rho \geq 0$,
\begin{equation}
  \label{eq:weak_duality}
  g^{AL}(b; \lambda, \rho) \leq g(b).
\end{equation}
Moreover, on the MILP setting, we have exact duality~\cite[Theorem 4, p. 381]{ahmed2017ald}
if the augmenting function $\psi$ is a \emph{norm}.
More precisely:
\begin{thm}
If the set $X$ is a rational polyhedron with integer constraints,
if the problem data $A$, $b_0$ and~$c$ are rational
and if for the given $b_0$ the problem is feasible with bounded value $g(b_0)$,
then there is a finite $\rho^*$ such that
\[
  \sup_{\lambda} g^{AL}(b_0; \lambda, \rho^*) = g(b_0).
\]
In addition, one can choose a \emph{finite} Lagrange multiplier $\lambda^*$
that attains the supremum.
\end{thm}

This motivates the introduction of \emph{augmented Lagrangian cuts}
using norms as the augmentation function.
Indeed, expanding the definition of $g^{AL}$
in the weak duality equation~\eqref{eq:weak_duality},
we get for any $b$:
\begin{align*}
  g(b) & \geq g^{AL}(b; \lambda, \rho) \\
       &   =  \min_{x \in X} \quad c^\top x - \lambda^\top(Ax - b) + \rho\norm{Ax - b} \\
       &   =  \min_{x \in X} \quad c^\top x - \lambda^\top(Ax - b_0 + b_0 - b) + \rho\norm{Ax - b_0 + b_0 - b}.
\end{align*}
By the triangular inequality,
\[
  \norm{Ax - b} \geq \norm{Ax - b_0} - \norm{b - b_0},
\]
so that
\begin{align*}
  g(b) & \geq \min_{x \in X} \quad c^\top x - \lambda^\top(Ax - b_0) + \lambda^\top(b_0 - b) + \rho\norm{Ax - b_0} - \rho\norm{b_0 - b} \\
       &   =  g^{AL}(b_0; \lambda, \rho) + \lambda^\top(b_0 - b) - \rho\norm{b_0 - b}.
\end{align*}
Therefore, calculating the augmented Lagrangian at $b_0$ provides a lower Lipschitz estimate for $g(b)$:
\begin{defi}
If $g$ is the optimal value function for a MILP, and given $b_0$, $\lambda$ and $\rho > 0$,
the \emph{augmented Lagrangian cut} centered at $b_0$ is the function
\[
  ALC_{b_0,\lambda,\rho}(b) := g^{AL}(b_0; \lambda, \rho) + \lambda^\top(b_0 - b) - \rho\norm{b_0 - b}.
\]
\end{defi}

Moreover, the exactness result above shows that
there exists a sufficiently large $\rho$ and appropriate $\lambda$
for which $g^{AL}(b_0; \lambda, \rho) = g(b_0)$, so the cut is \emph{tight} at $b_0$.
However, this proof stills leaves open the question of whether the family of such cuts
is \emph{uniformly} Lipschitz,
since as seen in the fractional-part example in figure~\ref{fig:milp_disc}
one might need arbitrarily large $\rho$ near discontinuities of the value function.

The Lipschitz setting that we assume for the value function
allows us to bypass this difficulty:
choosing $\rho = Lip(g)$ and $\lambda = 0$ produces a valid and tight cut,
so this provides an absolute upper bound on the needed $\rho$ for exact augmented Lagrangian duality.
This is why we proceed to characterize some MILPs with Lipschitz value functions,
before moving to the multistage case.

\subsection{MILPs with Lipschitz value functions}\label{subsec:MIP_optfunc}

As we have seen, the Lipschitz property of the value functions
was an important piece in the proof of convergence of the reverse-norm method
and also in the uniform bounds on $\rho$ for the augmented Lagrangian cuts.
It will also be fundamental in the analysis of the SLDP algorithm,
and therefore we present in this section
a sufficient condition that guarantees Lipschitz continuity
for the optimal value function of MILPs.

First, we show that the optimal value function for a Lipschitz objective
over a family of polyhedra is still Lipschitz.
Then, by enumerating the polyhedra over the realizations of integer variables,
we will arrive at the \emph{complete continuous recourse} (CCR) condition
that ensures the optimal value function of a MILP is Lipschitz.

\medskip

Consider the parameterized optimization problem
\begin{equation}\label{eq:lip_ovf}
  \begin{array}{rl}
    \nu(b) = \min_x & f(x) \\
    \text{s.t.} & (x,b) \in P.
  \end{array}
\end{equation}
for a Lipschitz function $f$ and a polyhedron $P$.
In order to analyze the function $\nu(b)$,
we have to understand the effect of the variations on the feasible set~$P(b)$
with respect to the parameter~$b$,
and compound this effect with the Lipschitz function~$f$.
For the first part, the Hoffman Lemma below provides the answer,
since it states that the symmetric difference between 
any two sets~$P(b)$ and~$P(b^*)$ is bounded by a ball centered
at the origin with a radius proportional to the norm~$\|b- b^*\|$.
This result resembles a version for sets of the Lipschitz 
continuity definition.

\begin{lemma}[Hoffman lemma~\cite{shapiro2014lectures}]
	Let $S(b)$ be a polyhedron parameterized by the right-hand side vector
	$b \in \mathbb{R}^m$ of a given linear system,  that is, 
	$S(b) = \{x \in \mathbb{R}^d \mid A x \leq b\}$.
	Let $b^* \in \mathbb{R}^m$ be a vector such that $S(b^*)$ is nonempty.
	Then, there exists $r > 0$ depending only on~$A$ such that
	\[
	S(b) \subseteq S(b^*) + r\cdot \|b - b^*\|\cdot \Delta, 
	\]
	where~$\Delta$ is the unit ball, i.e., 
	$\Delta = \{\eps \in \mathbb{R}^d \mid \|\eps \| \leq 1 \}$.
\end{lemma}

With this, we can guarantee that the optimal value function
of a minimization problem of a Lipschitz function over a given polyhedron
is also Lipschitz in its essential domain.
The same holds for a maximization problem because if~$f$ is a
Lipschitz function then so is~$-f$,
then, switching from maximization to minimization,
Theorem~\ref{thm:lips_cost-to-go} allows us to conclude the same result
about the corresponding optimal value function~$-\nu(\cdot)$.

\begin{thm}[Lipschitz cost-to-go functions]\label{thm:lips_cost-to-go}
	Let $\nu(\cdot)$ be the optimal value function defined by
	\begin{equation}\label{eq:opt_lipschitz}	 
	\begin{array}{rl}
	\nu(b) = 
	         \min & f(x) \\
	\textrm{s.t.} & (x,b) \in P, 
	\end{array}
	\end{equation}	 
	where the set $P$ is a polyhedron in $\mathbb{R}^{d+m}$,
	the function~$f(\cdot)$ is Lipschitz continuous with constant~$L$,
	and we assume that there is one value of $b$ such that $\nu(b)$ is finite.
	Then, the essential domain of~$\nu(\cdot)$, $\dom(\nu)$, 
	is a polyhedron, and the function~$\nu(\cdot)$
	restricted to~$\dom(\nu)$ is Lipschitz continuous with
	constant~$L \cdot \widetilde{r}$,
	where~$\widetilde{r}$ is a constant that depends only on~$P$.
\end{thm}
\begin{proof}
	First, we prove that~$\dom(\nu)$ is a polyhedron.
	Recall that~$\dom(\nu)$ is the set of vectors
	$b \in \mathbb{R}^m$ for which the problem~\eqref{eq:opt_lipschitz}
	is feasible, that is,
	$\dom(\nu) = 
	\left\{b \in \mathbb{R}^m 
	\ \middle| \ \exists x \in \mathbb{R}^d; 
	(x,b) \in P, \ f(x) < +\infty
	\right\}$.
	Since~$f$ is continuous, $f$ does not assume~$+\infty$ anywhere, 
	so $\dom(\nu)$ is the projection of~$P$ over the component~$b$:
	\[\dom(\nu) = \proj_b (P).\]
	As the image of a polyhedron by a linear map is also a polyhedron,
	we conclude that $\dom(\nu)$ is a polyhedron in $\mathbb{R}^m$.
	
	Now, we need to prove that~$\nu(\cdot)$ is Lipschitz continuous over~$\dom(\nu)$.
	Denote by~$W x + T b \leq h$ the linear constraint that defines~$P$, that is, 
	$P = \{(x,b) \in \mathbb{R}^{d+m} \mid W x + T b \leq h \}$, 
	and let~$S(u)$ be the set given by the linear system~$W x \leq u$.	
	Now, let $b_1$ and $b_2$ be two points in the domain of $\nu$.
	Taking a feasible point $(x_1,b_1) \in P$ for the problem defined by $b_1$,
	and applying the Hoffman Lemma for~$u := h - T b_2$,
	we get that there is a feasible point $(x_2,b_2) \in P$ such that
	\[
	  \norm{x_2 - x_1} \leq r \norm{(h - Tb_1) - (h - Tb_2)} \leq r \norm{T} \cdot \norm{b_2 - b_1}.
	\]
	Therefore, by the Lipschitz hypothesis on $f$,
	\[
	  \nu(b_2) \leq f(x_2) \leq f(x_1) + L r \norm{T} \cdot \norm{b_2 - b_1}.
	\]
	Taking the infimum over $x_1 \in P(b_1)$, we see that
	$\nu(b_2) \leq \nu(b_1) + L r \norm{T} \cdot \norm{b_2 - b_1}$.
	If $\nu(b_2)$ is not $-\infty$, this shows that $\nu(b_1) > -\infty$ as well,
	so $\nu$ never assumes the value $-\infty$.

	Finally, by symmetry, we obtain $\abs{\nu(b_2) - \nu(b_1)} \leq L r \norm{T} \cdot \norm{b_2 - b_1}$,
	which is the Lipschitz condition for $\nu$.
\end{proof}

To handle the Stochastic MILP case, it would be convenient if
the minimum of Lipschitz functions over a union of polyhedra
was Lipschitz as well.
Indeed, one could split the optimization variable $x = (y,z)$
over the integer~$z$ and continuous variables~$y$
and obtain
\[
  \nu(b) = \min_{z} \min_{y \in P(b,z)} f(y, z).
\]
Since each function $\nu_z(b) = \min_{y \in P(b,z)} f(y, z)$
is Lipschitz continuous by Theorem~\ref{thm:lips_cost-to-go} above,
we'd be done.
However, this is not true in general,
because the domains of each $\nu_z$ may be different.

Indeed, we illustrate in figures~\ref{subfig:milp_const_1}
and~\ref{subfig:milp_const_2} an example of a discontinuous  
optimal value function induced by the problem of minimizing a 
linear objective function over the union of two polyhedra. 
In this particular example, the feasible set is the intersection 
between the blue dashed line and the union of both 
vertical and horizontal rectangles, while the objective function 
is a linear function that decreases as the solution candidate moves to the left.
We show in figures~\ref{subfig:milp_const_1} and~\ref{subfig:milp_const_2} 
the optimal solution of this problem for two 
different right-hand side parameters, which control the height of the dashed line.
As that parameter changes, the dashed line moves up or down,
and the optimal solution changes abruptly as soon as a 
point in the horizontal rectangle becomes feasible, 
as occurs in figure~\ref{subfig:milp_const_2}.
This shows that the optimal value function is discontinuous, so 
it cannot be Lipschitz continuous.

\begin{figure}[!hbtp]
	\centering
	\myfig{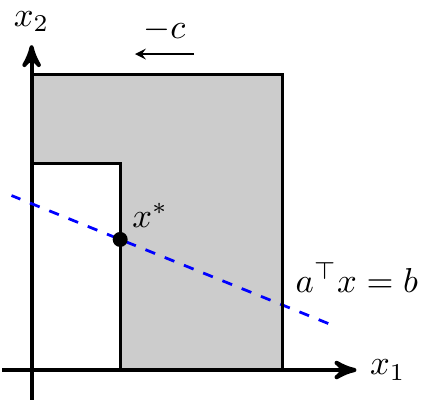}{Optimal solution with given $b$.}{milp_const_1}
	\hfill
	\myfig{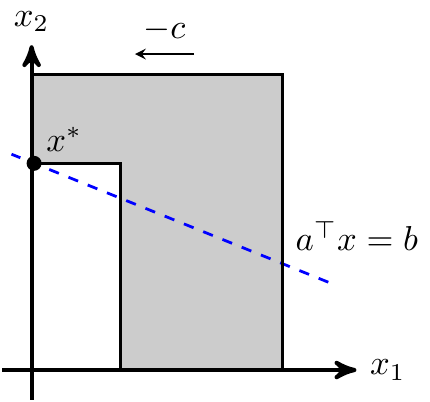}{Discontinuous effect in $b$.}{milp_const_2} 
	\caption{Minimum over union of polyhedra can be discontinous.}
\end{figure}%

In order to guarantee the Lipschitz condition for the optimal value function, 
we need to assume an additional hypothesis about the union of polyhedra that 
defines the feasible set of the optimization problem with Lipschitz objective.
Our typical optimization problem is
\begin{equation}\label{eq:opt_polyunion}
\begin{array}{rl}
\nu(b) = 
\underset{y}{\min} & f(y) \\
	 \textrm{s.t.} & (y,b) \in \bigcup_{i \in I} P_i,
\end{array}
\end{equation}	 
where $I$ is a finite index set, and $P_i$ is a polyhedron for each~$i \in I$. 
One sufficient condition for the Lipschitz continuity of~$\nu(b)$ 
is to assume that~$f$ 
is a Lipschitz continuous function in~$\proj_y(\bigcup_{i \in I} P_i)$
and that~$\proj_b (P_i)$ equals $\mathbb{R}^m$ for each $i \in I$.
This is called the \emph{Complete Continuous Recourse} (CCR) condition,
compare~\cite[Definition 1]{SDDiP2016}.
Indeed, under the CCR assumption, each optimal value 
function~$\nu_i(b)$ defined by the optimization problem
\begin{equation}\label{eq:opt_polypart}
\begin{array}{rl}
\nu_i(b) = 
\underset{y}{\min} & f(y) \\
	 \textrm{s.t.} & (y,b) \in P_i
\end{array}
\end{equation}	 
is Lipschitz continuous in $\mathbb{R}^m$
(Theorem~\ref{thm:lips_cost-to-go}) and since~$\nu(b)$ equals
$\min_{i \in I} \nu_i(b)$ we conclude that~$\nu(b)$ is also  
Lipschitz continuous in~$\mathbb{R}^m$.

\section{The Stochastic Lipschitz Dynamic Programming algorithm}\label{sec:lddp}
In this section, we present the Stochastic Lipschitz Dynamic Programming (SLDP)
algorithm for the multistage case in two different approaches:
the full scenario and the sampled scenario case.
In the full scenario approach, all the nodes are visited in the forward
and backward steps, and Lipschitz cuts are added for every
expected cost-to-go function.
In contrast, in the sampled scenario approach,
just the sampled scenarios are visited in the forward and backward steps,
and Lipschitz cuts are added only for the expected cost-to-go functions
of the sampled nodes.
We prove convergence to an optimal policy for the full scenario case, 
and we introduce an additional procedure in the sampled scenario case
to ensure convergence towards an~$\eps$-optimal policy.

\subsection{Multistage setting and Lipschitz continuity of cost-to-go functions}

For dealing with the stochastic multistage setting,
we fix some notation for describing the scenario tree.
Let $\mathcal{N}$ be the set of nodes,
where $1$ is the root node, and
$a:\mathcal{N}\backslash\{1\}\rightarrow\mathcal{N}$ is the ancestor function
that associates each node~$n$ except the root
to its ancestor node~$a(n)\in\mathcal{N}$. 
We also define the set~$\mathcal{S}(n)$ of successor nodes as the 
set of nodes with ancestor~$n$, that is, 
$\mathcal{S}(n) = \{m \in\mathbb{N}\mid  a(m) = n \}$, and for
each successor~$m$ there is a transition probability denoted by~$q_{nm}$ 
from node~$n$ to~$m$.
From $1$ and $\mathcal{S}$,
we define the \emph{stage} $t(n)$ of a node~$n\in\mathbb{N}$:
the root node belongs to stage~$1$,
and inductively the nodes in $S(n)$ belong to the stage $t(n)+1$.
The set of all nodes in stage~$\tau$ is denoted by~$\mathcal{N}_\tau$.

In the dynamic programming formulation of stochastic optimization,
we have a state variable~$x_n$ and a mixed integer control variable~$y_n$,
ranging over a feasible set $X_n$ and incurring an immediate cost $f_n(x_n,y_n)$.
As it is standard, we introduce a copy variable~$z_n$ that
carries the information from the previous state, so that the
cost-to-go and expected cost-to-go functions~$Q_n(\cdot)$ 
and~$\overline{Q}_n(\cdot)$ of each node~$n\in\mathcal{N} \backslash \{1\}$
satisfy the following recursive relationship:
\begin{align}\label{eq:multi_orig}
&\begin{array}{rl}
Q_n(x_{a(n)}) = \underset{x_n,y_n,z_n}{\min} & f_n(x_n,y_n) + \overline{Q}_{n}(x_n) \\
\text{s.t.} & (x_n,y_n,z_n) \in X_{n}, \\
& z_n = x_{a(n)},
\end{array}\\\label{eq:ectg_true}
&\overline{Q}_{n}(x_n) = 
\begin{cases}
0 & \text{if $\mathcal{S}(n) = \emptyset$,} \\
\sum_{m \in \mathcal{S}(n)} q_{nm} \cdot Q_{m}(x_n)  & \text{otherwise.}
\end{cases}
\end{align}
The nodes $n$ without successor are called \emph{leaf} nodes of the tree,
and they correspond to the last decision to be taken in the planning horizon.
Also, observe that the root node $1 \in \mathcal{N}$ does not have an ancestor,
so we can still define $\Qbar_1(x_1)$ by~\eqref{eq:ectg_true}, but
its stage problem~\eqref{eq:multi_orig} should be written as
\[
  \begin{array}{rl}
    Q_1 = \underset{x_1,y_1}{\min} & f_1(x_1,y_1) + \overline{Q}_1(x_1) \\
    \text{s.t.} & (x_1,y_1) \in X_1.
  \end{array}
\]
However, in order to avoid having to single out this special case,
we slightly abuse notation by fixing $0 = a(1)$, $x_0 = 0$,
and extend $X_1$ to a further dimension $z_1$.

From our discussion in the previous section,
we assume that $f_n$ is Lipschitz continuuos with constant $L_n$
and that $X_n$ satisfies the complete continuous recourse condition.
Under those hypothesis,
both the cost-to-go functions~$Q_n(\cdot)$,
and the expected cost-to-go functions~$\overline{Q}_n(\cdot)$
are Lipschitz continuous.

\begin{prop}[Stochastic multistage MILP programs]\label{cor:lips_multi}
Consider the stochastic multistage MILP program defined by~\eqref{eq:multi_orig}
and suppose that for every node~$n\in\mathcal{N}$ the cost-to-go function~$Q_n$ 
is not equal to $-\infty$ in any point, i.e.,~$Q_n(\cdot) > -\infty$, 
and the CCR condition holds for the feasible set~$X_n$. 
Then, the expected cost-to-go function~$\overline{Q}_n(\cdot)$ is 
Lipschitz continuous in~$\mathbb{R}^{d_n}$ with Lipschitz constant at most
\begin{equation}\label{eq:Lipschitz_estimates}
\widetilde{L}_{n} =
\begin{cases}
0 & \text{, if $\mathcal{S}(n) = \emptyset$,} \\
\sum_{m \in \mathcal{S}(n)} q_{nm} \cdot \big(L_m + \widetilde{L}_m \big) \cdot r_m & \text{, otherwise}
\end{cases}
\end{equation}
where $r_m$ is a constant that depends only on~$X_m$.
\end{prop}
\begin{proof}
	We proceed by backward induction on the scenario tree.
	For the leaf nodes, the statement holds by definition,
	since~$\overline{Q}_n(\cdot)$ is  identically zero.

	So, suppose this result holds for all successor nodes $m \in \mathcal{S}(n)$,
	and let's prove that it also holds for node~$n$.
	By the induction hypothesis, the expected cost-to-go 
	functions~$\overline{Q}_m(\cdot)$ are Lipschitz with constant~$\widetilde{L}_m$,
	and from Theorem~\ref{thm:lips_cost-to-go} each~$Q_m(\cdot)$ is Lipschitz with
	constant~$(L_m + \widetilde{L}_m)\cdot r_m$, where 
	$L_m$ is the Lipschitz constant of the objective function~$f_m$ 
	and $r_m$ is a constant from the Hoffman Lemma that only depends 
	on~$X_m$.  
	Since the expected value of Lipschitz functions is also 
	Lipschitz with constant equal to the expected value constant,
	the induction step is proved.
\end{proof}

Since problem~\eqref{eq:multi_orig} for each node
admits the Lipschitz decomposable structure of~\eqref{eq:determ_main},
one could imagine using the reverse-norm method of Algorithm~\ref{alg:cutting_tents},
or augmented Lagrangian cuts, to approximate its solution.
However, in the multistage case we lack one fundamental property we used,
namely that we can compute exactly the `complex' function $g(x)$,
which in this case is $\overline{Q}_n(x_n)$.
Indeed, we're only able to produce lower approximations for it,
and the next sections will deal with the necessary estimates
to prove convergence under this weaker hypothesis.

\subsection{Approximating the value functions}
Before we present the SLDP algorithm,
we need to introduce some notation
for the approximations along the iterations of the algorithm.
As usual, we denote by~$\overline{\mathfrak{Q}}_n^k(x_n)$ the expected cost-to-go
approximation induced by the Lipschitz cuts
at iteration~$k$.
For the purpose of convergence analysis, we
consider the approximations $Q_n^k(x_{a(n)})$ and~$\overline{Q}_n^k(x_{a(n)})$ 
of the cost-to-go and expected cost-to-go functions at iteration~$k$
defined below:
\begin{align}\label{eq:ctg_approx}
&\begin{array}{rl}
Q_n^k(x_{a(n)}) = \underset{x_n,y_n,z_n}{\min} & f_n(x_n,y_n) + \overline{\mathfrak{Q}}_{n}^{k}(x_n) \\
\text{s.t.} & (x_n,y_n,z_n) \in X_n, \\
& z_n = x_{a(n)},
\end{array}\\\label{eq:ectg_approx}
&\overline{Q}_{n}^k(x_n) = 
\begin{cases}
0 & \text{if $\mathcal{S}(n) = \emptyset$,} \\
\sum_{m \in \mathcal{S}(n)} q_{nm} \cdot Q_{m}^k(x_n)  & \text{otherwise}
\end{cases}
\end{align}
We assume we are given, for the first iteration,
a Lipschitz lower approximation~$\Qfrak_n^1(\cdot)$
of the expected cost-to-go function~$\overline{Q}_n(\cdot)$
for each node $n$ in the tree.
In practice, the first expected cost-to-go approximations
are identically zero, since costs are usually non-negative.

Then, we update the expected cost-to-go
approximation~$\overline{\mathfrak{Q}}_n^k(x_n)$ of iteration~$k$ at 
a given point~$x_n^k$ using the reverse-norm
cut~$\overline{Q}_n^{k+1}(x_{n}^k) - \rho_n \cdot \|x_{n} - x_{n}^k\|$:
\begin{equation}\label{eq:cut_hypo}
\overline{\mathfrak{Q}}_n^{k+1}(x_{n}) = \max\left\{\overline{\mathfrak{Q}}_n^{k}(x_{n}),\ \overline{Q}_n^{k+1}(x_{n}^k) - \rho_n \cdot \|x_{n} - x_{n}^k\| \right\},
\end{equation}
where~$\rho_n > 0$ is any constant greater than or equal to the Lipschitz
constant~$\widetilde{L}_n$ defined on~\eqref{eq:Lipschitz_estimates}. 
Note that we have 
used~$\overline{Q}_n^{k+1}(x_{n}^k)$ instead of~$\overline{Q}_n^{k}(x_{n}^k)$
in the cut update~\eqref{eq:cut_hypo} because the Lipschitz cuts of the SLDP
are updated from the last to the first stage, so all 
expected cost-to-go approximations~$\overline{\mathfrak{Q}}_m^{k}(x_{m})$
of the successor nodes~$m$ of $n$ are updated
to~$\overline{\mathfrak{Q}}_m^{k+1}(x_{m})$ before the computation
of the optimal value~\eqref{eq:ctg_approx} with state $x_{a(m)} = x_n^k$.
So, given each node~$m$ and iteration~$k$ we obtain in the backward step 
the cost-to-go approximation $Q_m^{k+1}(x_{a(m)})$ evaluated at~$x_n^k$,
and since the expected cost-to-go approximation is the corresponding
weighted average, we obtain~$\overline{Q}_n^{k+1}(x_{n}^k)$ for the
Lipschitz cut~\eqref{eq:cut_hypo}.

There are some important comments about the concepts introduced so far. 
First, we note that the 
sequence ~$\{\overline{\mathfrak{Q}}_n^{k}\}_{k\in\mathbb{Z}_+}$ 
is a non-decreasing sequence of functions, 
and since it belongs to the objective function of~\eqref{eq:ctg_approx}, 
we conclude that the 
sequence of cost-to-go function
approximations~$\{Q_n^{k}\}_{k\in\mathbb{Z}_+}$ 
is also non-decreasing.
Second, the expected cost-to-go
approximation~$\overline{Q}_{n}^k(x_n)$ defined in~\eqref{eq:ectg_approx}
is a weighted average of non-decreasing functions~${Q}_{m}^k$, so 
the corresponding sequence~$\{\overline{Q}_{n}^k\}_{k\in\mathbb{Z}_+}$ 
is also non-decreasing.
Third, the function~$\overline{Q}_n^k$ plays an important role
in the convergence analysis of the SLDP since the cuts
of~\eqref{eq:cut_hypo} are tight for~$\overline{Q}_n^{k+1}(x_n)$ 
at the forward solution~$x_n^k$.
Last, the quality of the expected cost-to-go
approximation~$\overline{Q}_n^{k}$ of a given node~$n\in\mathcal{N}$ 
depends on the quality of those approximations at the successor 
nodes~$m \in\mathcal{S}(n)$, so this explains the reason of computing
Lipschitz cuts from the last to first stage. 

We will prove convergence of the full scenario (resp. sampled scenario) SLDP algorithm
by proving that the sequence of feasible
policies~$\big(x_n^k\big)_{n\in\mathcal{N}}$ produced by the algorithm
converges to an optimal ($\eps$-optimal) policy~$\big(x_n^*\big)_{n\in\mathcal{N}}$.
As in Lemma~\ref{lemma:cluster_converg}, we show that
$\overline{Q}_n^{k+1}(x_n^k)$ converges to
(and $\eps$-approximation of)~$\overline{Q}_n(x_n^*)$
where~$\overline{Q}_n$ is the \emph{true} expected cost-to-go
function~\eqref{eq:ectg_true}.
In the examples of section~\ref{sec:examples},
we will see that
both expected cost-to-go approximations~$\overline{Q}_n^{k}$
and~$\overline{\mathfrak{Q}}_n^{k}$ approximate the true
expected cost-to-go function~$\overline{Q}_n$ in a neighborhood of the optimal ($\eps$-optimal) policy
solution~$x_n^*$, however those approximations are usually poor elsewhere.

\medskip

To simplify our statements and make the logic in the proofs easier to follow,
we will assume that the starting lower bounds
$\Qfrak_n^1$ are also valid lower bounds for $\Qbar_n^1$,
which imposes a compatibility constraint between
$\Qfrak_n^1$ and its successors $\Qfrak_m^1$ for $m \in \mathcal{S}(n)$.
If one doesn't have this property,
then only the inequalities
$\Qfrak_n^k \leq \Qbar_n$ and $\Qbar_n^k \leq \Qbar_n$
are ensured in the following Lemma.
Again, in common situations where costs are positive
and all $\Qfrak_n^1 = 0$, this is immediately satisfied.

By the definition of cost-to-go and expected cost-to-go approximations,
they form a monotone sequence of valid
lower bounds for the true expected cost-to-go functions:
\begin{lemma}\label{lemma:ctg_desig}
	Consider the stochastic multistage MILP program~\eqref{eq:multi_orig} 
	satisfying the CCR condition, and let~$Q_n^k$, $\overline{Q}_n^k$
	and $\overline{\mathfrak{Q}}_n^k$ be the cost-to-go and expected 
	cost-to-go approximations of~\eqref{eq:ctg_approx}, \eqref{eq:ectg_approx}
	and~\eqref{eq:cut_hypo}.
	Then
	\begin{equation}\label{eq:approx_ineq}
	\overline{\mathfrak{Q}}_n^k(\cdot) \leq \overline{Q}_n^k(\cdot) \leq \overline{Q}_n(\cdot),
	\end{equation}
	for every node~$n \in \mathcal{N}$ and iteration~$k \in \mathbb{Z}_+$.
\end{lemma}
\begin{proof}
	We proceed by backward induction on the tree.
	For leaf nodes, inequality~\eqref{eq:approx_ineq} holds because
	$\overline{Q}_n^k$ and $\overline{Q}_n$ are identically zero,
	by definition.
	Suppose that inequality~\eqref{eq:approx_ineq} holds
	for every successor node $m \in \mathcal{S}(n)$ at iteration $k$.
	By the induction hypothesis, 
	the function~$\overline{\mathfrak{Q}}_m^k$ is less than
	or equal to~$\overline{Q}_m$, and by the optimization problems~\eqref{eq:multi_orig} and~\eqref{eq:ctg_approx} 
	we conclude that the cost-to-go
	approximation~$Q_m^k(\cdot)$ is less than or equal to the 
	true cost-to-go function~$Q_m$.
	Since we guarantee this property for every successor node~$m$, 
	we get the same inequality for their respective weighted
	averages~$\overline{Q}_n^k$ and~$\overline{Q}_n$.
	
	Now, let's prove that~$\overline{\mathfrak{Q}}_n^k$
	is less than or equal to~$\overline{Q}_n^k$ by induction on the 
	iteration~$k$.
	In the first iteration, the cost-to-go
	approximation~$\overline{\mathfrak{Q}}_n^1$ is less than or
	equal to~$\overline{Q}_n^1$ by hypothesis.
	Suppose that~$\overline{\mathfrak{Q}}_n^j$
	is less than or equal to~$\overline{Q}_n^j$ for every iteration~$j$ less than~$k$.
	We will prove that such inequality also holds
	for iteration~$k$.
	Indeed, by the induction hypothesis and the non-decreasing property 
	of~$\{\overline{Q}_n^k\}_{k\in\mathbb{Z}_+}$, we have the following inequalities: 
	\[
	\overline{\mathfrak{Q}}_n^{k-1}(\cdot) \leq \overline{Q}_n^{k-1}(\cdot)
	\leq \overline{Q}_n^{k}(\cdot).
	\]
	Using the updating formula~\eqref{eq:cut_hypo}, we conclude 
	that~$\overline{\mathfrak{Q}}_n^k$ is less than or equal to~$\overline{Q}_n^k$
	because the reverse-norm cut is also a lower bound for $\overline{Q}_n^k$.
\end{proof}

Throughout this paper we assume the CCR condition for the 
true stochastic multistage MILP program~\eqref{eq:multi_orig}. 
Additionally, we also require the set of feasible 
policy's states~$\proj_x X_n$ to be compact,
and we name the resulting assumption as the
\emph{Compact State Complete Continuous Recourse} (CS-CCR) condition.

\subsection{Full scenario approach}

The SLDP algorithm for the full scenario approach 
is analogous to the Nested Cutting Plane algorithm, 
but with Lispchitz cuts instead of linear cuts. 
As described in Algorithm~\ref{alg:full_ald},
starting from a valid lower bound $M_n$ for all cost-to-go functions,
and an upper bound $\rho_n$ for their Lipschitz constants,
we improve the lower bounds
near the candidate optimal solutions of each iteration.
Thus, in the forward step, the full scenario SLDP algorithm solves the optimization
problems~\eqref{eq:ctg_approx} from the root to the leaves,
that is, in ascending order of stages,
and obtains feasible state and control variables~$(x_n,y_n,z_n) \in X_n$
for each node of the scenario tree.
Then, in the backward step, it updates from the leaves to the root
the expected cost-to-go approximation~$\overline{\mathfrak{Q}}_n^k$
using formula~\eqref{eq:cut_hypo} and the Lipschitz cuts centered 
at the states obtained in the forward step.

\begin{algorithm}[!hbt]
	\caption{Full scenario SLDP}\label{alg:full_ald}
		\begin{flushleft}
		\textbf{Input:} lower bound~$\overline{\mathfrak{Q}}_n^1 \equiv M_n$,
		cut constant~$\rho_n \geq L_n$.\\
		\textbf{Output:} policy~$(x,y,z) \in \mathbb{X}$ and 
		expected cost-to-go~$\{\overline{\mathfrak{Q}}_n^k\}_{n\in\mathcal{N}}$
		approximations.
	\end{flushleft}	
	Set initial iteration $k=1$ and go to step 1
	\begin{algorithmic}[1]
		\State Forward step: set $n = 1$, $x_0^k = 0$. 
		For each $t$ from $1$ to $T-1$, and for each $n \in \mathcal{N}_t$ do:
		\begin{enumerate}
			\item[(a)] Solve the problem corresponding to $Q_n^k(x_{a(n)}^k)$ and obtain $(x_n^k,y_n^k,z_n^k)$:
			\[\begin{array}{rl}
			\underset{x_n,y_n,z_n}{\min} & f_n(x_n,y_n) + \overline{\mathfrak{Q}}_{n}^{k}(x_n) \\
			\text{s.t.} & (x_n,y_n,z_n) \in X_n, \\
			& z_n = x_{a(n)}^k;
			\end{array}\]						
		\end{enumerate}
		\State Backward step: For each $t$ from $T-1$ to $1$, and for each $n \in \mathcal{N}_t$ do:
		\begin{enumerate}
			\item[(a)] update 
			$\overline{\mathfrak{Q}}_n^{k+1}(x_{n}) = \max\left\{\overline{\mathfrak{Q}}_n^{k}(x_{n}),\ \overline{Q}_n^{k+1}(x_{n}^k) - \rho_n \cdot \|x_{n} - x_{n}^k\| \right\}$, where			
			\begin{align*}
			&\ \, \overline{Q}_{n}^{k+1}(x_n^k) = 				\sum_{m \in \mathcal{S}(n)} q_{nm} \cdot Q_{m}^{k+1}(x_n^k), \text{ and }\\
			&\begin{array}{rl}
			Q_m^{k+1}(x_{n}^k) = \underset{x_m,y_m,z_m}{\min} & f_m(x_m,y_m) + \overline{\mathfrak{Q}}_{m}^{k+1}(x_m) \\
			\text{s.t.} & (x_m,y_m,z_m) \in X_m, \\
			& z_m = x_n^k;
			\end{array}
			\end{align*}						
		\end{enumerate}
		\State Increase $k$ by $1$, and go to step $1$.
	\end{algorithmic}
\end{algorithm}

We have not provided a stopping criterion for Algorithm~\ref{alg:full_ald}.
Although in the full scenario case we could have chosen a criterion
equivalent to the one in Algorithm~\ref{alg:cutting_tents},
in the sampled case one would need to
compute the optimal solution at every node of the scenario tree
to have a deterministic upper bound for the optimal policy,
which is unrealistic.
So, we preferred to emphasize the similarities between the full and sampled scenario cases.
and present their convergence results only in asymptotic form.

In order to simplify the notation and improve readability, 
we will assume that a variable $x$ or a vector $(x,y,z)$ that do not 
have the subscript~$n$ is the vector composed by the corresponding variables or vectors
for all nodes:
\begin{itemize}
	\item $x := (x_n)_{n\in\mathcal{N}}$;
	\item $(x,y,z) := \Big((x_n,y_n,z_n)\Big)_{n\in\mathcal{N}}$.
\end{itemize}
We refer to $(x,y,z)$ as a policy and~$x$ as the policy's states,
and we denote by~$\mathbb{X}$ the set of feasible policies and 
by $\proj_x \mathbb{X}$ the projection of~$\mathbb{X}$ in the policy's states.

By analogy with the proof of the (deterministic) reverse-norm method
in Lemma~\ref{lemma:cluster_converg} and Theorem~\ref{thm:cvg_rnm},
we start proving that
the expected cost-to-go approximation~$\overline{\mathfrak{Q}}_n^k$
approximates the true expected cost-to-go function~$\overline{Q}_n$
in a neighborhood of any cluster state
induced by the forward step.

\begin{lemma}\label{lemma:cluster_policy}
	Let $x^* \in \proj_x{\mathbb{X}}$ be a cluster point of
	the sequence of policy states~$\{x^k\}_{k\in\mathbb{N}}$
	generated by the forward step of Algorithm~\ref{alg:full_ald},
	and let $\mathcal{K}$ be the indices of a subsequence
	that converges to~$x^*$.
	Then $\{\overline{\mathfrak{Q}}_n^k(x_n^k)\}_{k\in\mathcal{K}}$ converges to $\overline{Q}_n(x_n^*)$,
	\begin{equation}\label{eq:induction_hypo}
	\lim\limits_{k\in\mathcal{K}} \overline{\mathfrak{Q}}_n^k(x_n^k) = \overline{Q}_n(x_n^*),
	\end{equation}
	for every node~$n\in\mathcal{N}$.
\end{lemma}
\begin{proof}
	Let~$\{(x^k,y^k,z^k)\}_{k\in\mathbb{N}}$ be the sequence of policies 
	obtained in the forward step of algorithm~\ref{alg:full_ald}.
	By the compactness assumption of~$\proj_x X_n$, the 
	set of feasible policy states~$\proj_x \mathbb{X}$ is also compact, so
	there is a subsequence of~$\{x^k\}_{k\in\mathbb{N}}$
	that converges to a cluster point~$x^* \in \proj_x\mathbb{X}$.
	Denote by~$\mathcal{K}$ the indices of this subsequence, that is,
	$\lim\limits_{k \in \mathcal{K}} x^k = x^*$.
	We will show that equation~\eqref{eq:induction_hypo} holds 
	by backward induction on the tree.
	It trivially holds for the leaf nodes,
	since both functions~$\overline{\mathfrak{Q}}_n$ and~$\overline{Q}_n$
	are identically zero, by hypothesis.

	Now, suppose that equation~\eqref{eq:induction_hypo} holds for
	every successor node~$m \in \mathcal{S}(n)$.
	From Lemma~\ref{lemma:ctg_desig}, we have an upper bound:
  \[
    \Qbar_n(x_n^k) \geq \Qfrak_n^k(x_n^k),
  \]
  which, by continuity of $\Qbar_n$, yields:
  \begin{equation}
    \label{eq:qbar_qfrak_ft}
    \Qbar_n(x_n^*) \geq \limsup \Qfrak_n^k(x_n^k).
  \end{equation}
  So we only need to prove that the lower approximations are large enough.

  As in Lemma~\ref{lemma:cluster_converg},
  we denote by~$j$ the index in~$\mathcal{K}$ immediately before~$k$.
  By monotonicity of the approximations,
  $\Qfrak_n^k$ is larger than all of the Lipschitz cuts constructed,
  in particular the one at iteration $j$.
  Therefore,
  \[
    \Qfrak_n^k(x_n^k)
    \geq \Qbar_n^{j+1}(x_n^j) - \rho_n \cdot \norm{x_n^j - x_n^k}.
  \]
  Note that, differently from Lemma~1,
  we don't obtain the exact expected cost-to-go function $\Qbar_n$,
  but only its approximation $\Qbar_n^{j+1}$.
  That's why our proof splits in two parts:
  one bounding the difference between $\Qbar_n^j$ and~$\Qfrak_n^k$,
  and the other bounding the one between $\Qbar_n$ and~$\Qbar_n^j$.
  Let's complete the first one, which we already started.
  Since $\Qbar_n^j$ is an increasing sequence, $\Qbar_n^j \leq \Qbar_n^{j+1}$, and we obtain
  \begin{equation}
    \label{eq:qfrak_qbark_ft}
    \Qfrak_n^k(x_n^k)
    \geq \Qbar_n^j(x_n^j) - \rho_n \cdot \norm{x_n^j - x_n^k}.
  \end{equation}

  To show that $\Qbar_n$ and $\Qbar_n^j$ are close,
  we use their definitions in~\eqref{eq:ectg_approx} and~\eqref{eq:ectg_true}:
	\begin{align} \nonumber
    \Qbar_n^j(x_n^j) - \Qbar_n(x_n^j)
    &   =  \sum_{m \in \mathcal{S}(n)} q_{nm}\cdot \left[Q_m^j(x_n^j) - Q_m(x_n^j) \right] \\
    \label{eq:induct_ineq}
    & \geq \sum_{m \in \mathcal{S}(n)} q_{nm}\cdot \left[\Qfrak_m^j(x_m^j) - \Qbar_m(x_m^j) \right],
	\end{align}
	where the inequality follows because~$(x_m^j,y_m^j,z_m^j)$
	are optimal solutions to~\eqref{eq:ctg_approx}
	and feasible solutions to~\eqref{eq:multi_orig} for each $m \in \mathcal{S}(n)$.
  Taking~\eqref{eq:qfrak_qbark_ft} and~\eqref{eq:induct_ineq} together
  and rearranging terms we get
  \[
    \Qfrak_n^k(x_n^k)
    \geq \Qbar_n(x_n^j) - \rho_n \cdot \norm{x_n^j - x_n^k}
         - \sum_{m \in \mathcal{S}(n)} q_{nm}\cdot \left[\Qbar_m(x_m^j) - \Qfrak_m^j(x_m^j) \right].
  \]

  Now, take the limit as $k$ goes to $\infty$,
  which also makes $j$ grow to $\infty$, and both $x^k$ and $x^j$ converge to $x^*$.
  Since the expected cost-to-go function~$\Qbar_n$ is continuous, we obtain:
  \[
    \liminf \Qfrak_n^k(x_n^k) \geq \Qbar_n(x_n^*)
  \]
  because both residual terms vanish in the limit,
  the second one going to zero by our induction hypothesis.
  Together with the upper bound from equation~\eqref{eq:qbar_qfrak_ft},
  this shows that the limit exists and concludes our proof.
\end{proof}

As a consequence of Lemmas~\ref{lemma:ctg_desig} and~\ref{lemma:cluster_policy},
the expected cost-to-go 
approximation~$\overline{Q}_n^k$ also 
approximates the true expected cost-to-go 
function~$\overline{Q}_n$ in a neighborhood of 
any cluster point of the sequence of feasible policy's states
induced by the forward step of the full scenario SLDP.
Using the argument that leads to
inequality~\eqref{eq:induct_ineq} in the proof of 
Lemma~\ref{lemma:cluster_policy} we get 
that the cost-to-go approximation~$Q_n^k$ 
also approximates the true cost-to-go 
function~$Q_n$ in a neighborhood of any cluster policy state. 
That is, the following limits also hold:
\begin{align*}
\lim\limits_{k\in\mathcal{K}}\overline{Q}_{n}^{k}(x_n^k) & = \overline{Q}_n(x_n^*), \\
\lim\limits_{k\in\mathcal{K}}Q_{n}^{k}(x_{a(n)}^k) & = Q_n(x_{a(n)}^*),
\end{align*}
for any convergent subsequence~$\{x^k\}_{k\in\mathcal{K}}$
of policy states induced by the forward
step of the SLDP algorithm, and $x^* \in \proj_x\mathbb{X}$
the corresponding limit point.

\begin{thm}
  \label{thm:cvg_sldp_ft}
	The sequence of lower 
	bounds~$\{Q_1^k\}_{k\in\mathbb{N}}$ 
	induced by the SLDP algorithm~\ref{alg:full_ald} converges
	to the optimal value~$Q_1$ of the true 
	stochastic multistage MILP program~\eqref{eq:multi_orig},
	and every cluster point of the sequence of feasible
	policies~$\{(x^k,y^k,z^k)\}_{k\in\mathbb{N}}$
	generated by the forward step of Algorithm~\ref{alg:full_ald} 
	is an optimal policy.
\end{thm}
\begin{proof}	
	Let~$\{(x^k,y^k,z^k)\}_{k\in\mathbb{N}}$ be the sequence
	of feasible policies generated by the full scenario SLDP
	algorithm~\ref{alg:full_ald}.
	Let $\mathcal{K}$ be the 
	set of indices of a convergent subsequence
	of policy states~$\{x^k\}_{k\in\mathbb{N}}$, and
	let~$x^*$ be the corresponding limit point,
	which exists by compactness of $\proj_x \mathbb{X}$.
	As for equation~\eqref{eq:induct_ineq}, at the root node we have
	\begin{align*}
	0 \leq Q_1 - Q_1^k & \leq 
	\Big[f_1(x_1^k,y_1^k) +  \overline{Q}_1(x_1^k)\Big] 
	-\Big[f_1(x_1^k,y_1^k) +
	\overline{\mathfrak{Q}}_1^{k}(x_1^k)\Big], \\
	& = \overline{Q}_1(x_1^k) -
	\overline{\mathfrak{Q}}_1^{k}(x_1^k),
	\end{align*}
	because $(x_1^k,y_1^k,z_1^k)$ is a feasible solution to
	the optimization problem whose optimal value is $Q_1$
	and optimal solution to that whose optimal value is~$Q_1^k$. 
	Using Lemma~\ref{lemma:cluster_policy}, we conclude  
	the convergence of the subsequence~$\{Q_1^k\}_{k\in\mathcal{K}}$ to~$Q_1$.
	Since the whole sequence~$\{Q_1^k\}_{k\in\mathbb{N}}$ is non-decreasing,
	we get convergence to~$Q_1$.

	Now, suppose that there is a cluster point~$(x,y,z)$ of the sequence of
	feasible policies~$\{(x^k,y^k,z^k)\}_{k\in\mathbb{N}}$,
	and denote also by~$\mathcal{K}$ the set of indices of
	the corresponding subsequence.
	In order to prove that~$(x,y,z)$ is an optimal policy,
	we need to show that its components are optimal solutions
	to the optimization problem of each node~$n\in\mathcal{N}$ whose optimal value is~$Q_n(x_{a(n)})$.
	We will proceed by forward induction on the tree.
	Indeed, we have just shown that $(x_1,y_1,z_1)$ is
	an optimal solution at the root node.
	Now, assume that this result holds for the ancestor node $a(n)$.
	Using the same argument as before, we have the following inequalities:
	\begin{align*}
	f(x_n^k,y_n^k) + \overline{\mathfrak{Q}}_n^k(x_n^k) = Q_n^k(x_{a(n)}^k) 
	\leq Q_n(x_{a(n)}^k) \leq f(x_n^k,y_n^k) + \overline{Q}_n(x_n^k).
	\end{align*}
	So, taking the limit over~$\mathcal{K}$ on both
	sides of the inequality and using 
	Lemma~\ref{lemma:cluster_policy}, we conclude
	that $(x_n,y_n,x_{a(n)})$ is an optimal solution of
	the optimization problem whose optimal value 
	is~$Q_n(x_{a(n)})$.
\end{proof}

Just as it was the case for the proof of both Lemma~\ref{lemma:cluster_converg}
and theorem~\ref{thm:cvg_rnm},
we again just use the same properties of the reverse-norm cuts,
namely that they are uniformly Lipschitz
and that we are able to construct exact cuts at trial points,
for the \emph{approximate} future cost function $\Qbar_n^k$.
As before, this shows that any method of producing \emph{uniformly Lipschitz tight cuts}
in the nested form of stochastic optimization problems
will result in a convergent algorithm on the full scenario approach.
In particular, one can use the augmented Lagrangian cuts
from section~\ref{sec:ald_cuts} provided one takes $\rho_n$ large enough
that the resulting cut is exact.

\subsection{Sampled tree approach}
In multistage stochastic programming problems
with a reasonable number of stages,
it is computationally intractable to visit every node of the scenario tree.
So, one needs to sample paths on the scenario tree
and iteratively approximate the expected cost-to-go functions at each stage
to obtain a ``reasonable'' solution. 
In this paper, we focus on the sampling scheme of one random path
per forward iteration, but its conversion to more general
schemes is straightforward.

We emphasize that a path on the scenario tree
is chosen at random, so  a node~$n$ may not belong 
to the path of some forward step iterations.
Let~$\mathcal{J}_n$ be the set of iterations~$k$ of the Sampled-SLDP for which 
the path of the forward step contains the node~$n$.
Note that~$\mathcal{J}_n$ is a random set, since it depends on
each experiment~$\omega\in\Omega$, and the probability of node~$n$ being
draw an infinite number of times equals one, i.e.,
\[
\mathbb{P}(\{\omega \in \Omega \mid \# \mathcal{J}_n(\omega) = +\infty \}) = 1,
\] 
by the Borel-Cantelli Lemma.
We will assume a realization of the sampling where this is the case,
to avoid repeating ``with probability one'' in what follows.

Also, observe that the collection of
sets~$\{\mathcal{J}_m \mid m \in\mathcal{S}(n)\}$ induced by the successor
nodes covers~$\mathcal{J}_n$, that is
\[
\mathcal{J}_n = 
\bigcup_{m\in\mathcal{S}(n)} \mathcal{J}_m,
\]
since a path that contains a node~$n$ also contains 
some successor node~$m$.
In the deterministic case, the set of iterations~$\mathcal{J}_n$ 
equals~$\mathbb{Z}_+$ for every node~$n$, since all nodes
are visited in the forward step. 
In the analysis of the Sampled-SLDP algorithm, we need to refer
to optimal solutions of nodes that do not belong to a given
forward path,
even if in practice they are not computed.
We still use the same notation~$(x_n^k,y_n^k,z_n^k)$ to refer to an
optimal solution of node~$n$ and iteration~$k$.

Following the same organization of the previous sections,
we would like to prove that for each node~$n$
there is a subset~$\mathcal{K}_n$ of~$\mathcal{J}_n$
such that the following limit holds:
\begin{equation}\label{eq:induction_missing_hypo}
\lim\limits_{k\in\mathcal{K}_n} \overline{\mathfrak{Q}}_n^k(x_n^k) 
\to
\overline{Q}_n(x_n^*),
\end{equation}
where~$\{x_n^k\}_{k\in\mathcal{K}_n}$ is a subsequence 
of policy states converging to a limit point~$x_n^*$.
However, the main obstacle of this lemma 
is the induction step, since
we need to control the difference between~$\overline{Q}_n(x_n^k)$ 
and~$\overline{Q}_n^k(x_n^k)$ using inequality~\eqref{eq:induct_ineq}
or some variation, as in the proof of Lemma~\ref{lemma:cluster_policy}.
Inequality~\eqref{eq:induct_ineq} directly is not suitable for the proof,
because there we implicitly used that~$\mathcal{K}_m$ equals~$\mathcal{K}_n$
for every successor node~$m$
to be able to use the induction hypothesis~\eqref{eq:induction_missing_hypo}.

In order to ensure convergence of the Sampled-SLDP algorithm, 
we consider an additional step to stabilize the policy states
obtained in the forward step.
Instead of computing reverse norm cuts at 
every \emph{new} forward solution~$x_n^k$, 
we check if the new feasible point is more than $\delta > 0$ away
from all previous forward solutions $x_n^{1},\dots,x_n^{k-1}$.
If this is the case, then we update  
the expected cost-to-go function~$\overline{\mathfrak{Q}}_n^k(\cdot)$
with the reverse norm cut centered 
at the new policy state~$x_n^k$;
otherwise we improve it at the closest
previous forward solution~$x_n^j$, 
see Algorithm~\ref{alg:sample_ald}.
Note that after a finite number of iterations
the forward incoming state~$x_n^k$ becomes trapped in a finite number 
of possibilities, since node~$n$ will be visited an infinite number of times
and the set of feasible policy states~$\proj_x X_n$ is compact.
We also show in Lemma~\ref{lemma:finite_policy} that the expected cost-to-go 
approximation~$\overline{\mathfrak{Q}}_n^k$ converges in a finite
number of iterations to a Lipschitz function~$\overline{\mathfrak{U}}_n$, 
which is an $\eps$-approximation of the true expected
cost-to-go function~$\overline{Q}_n$ at any cluster point~$x_n^*$.

\begin{algorithm}[!htb]
	\caption{Sampled SLDP}\label{alg:sample_ald}
\begin{flushleft}
	\textbf{Input:} lower bound~$\overline{\mathfrak{Q}}_n^1 \equiv M_n$,
	cut constant~$\rho_n \geq L_n$.\\
	\textbf{Output:} policy~$(x,y,z) \in \mathbb{X}$ and 
	expected cost-to-go~$\{\overline{\mathfrak{Q}}_n^k\}_{n\in\mathcal{N}}$
	approximations.
\end{flushleft}	
	Set initial iteration $k=1$ and go to step 1	
	\begin{algorithmic}[1]
		\State Forward step: set $n = 1$, $x_0^k = 0$. 
		While $\mathcal{S}(n) \neq \emptyset$ do:
		\begin{enumerate}
			\item[(a)] Solve the problem corresponding to $Q_n^k(x_{a(n)}^k)$ and obtain $(u_n^*,y_n^*,z_n^*)$:
			\[\begin{array}{rl}
			\underset{u_n,y_n,z_n}{\min} & f_n(u_n,y_n) + \overline{\mathfrak{Q}}_{n}^{k}(u_n) \\
			\text{s.t.} & (u_n,y_n,z_n) \in X_n, \\
			& z_n = x_{a(n)}^k;
			\end{array}\]			
			\item[(b)] Compute the minimum 
			$\Delta_n^k = \displaystyle\min_{j=1,\dots,k-1}\|u_n^* - x_{n}^{j}\|$, 
			and denote by $i$ the minimizer index.
			If $\Delta_n^k < \delta$ then set $x_n^k = x_n^{i}$, 
			else set $x_n^k = u_n^*$;			
			\item[(c)] Sample $m \in \mathcal{C}(n)$, and set $n = m$;
		\end{enumerate}
		\State Backward step: Start with the particular node $n$ of the end of step 1.
		While $n$ is not the root $1$, do:
		\begin{enumerate}	
			\item[(a)] update 
			$\overline{\mathfrak{Q}}_n^{k+1}(x_{n}) = \max\left\{\overline{\mathfrak{Q}}_n^{k}(x_{n}),\ \overline{Q}_n^{k+1}(x_{n}^k) - \widetilde{L}_n \cdot \|x_{n} - x_{n}^k\| \right\}$, where
			\begin{align*}
			&\ \, \overline{Q}_{n}^{k+1}(x_n^k) = 				\sum_{m \in \mathcal{S}(n)} q_{nm} \cdot Q_{m}^{k+1}(x_n^k), \text{ and }\\
			&\begin{array}{rl}
			Q_m^{k+1}(x_{n}^k) = \underset{x_m,y_m,z_m}{\min} & f_m(x_m,y_m) + \overline{\mathfrak{Q}}_{m}^{k+1}(x_m) \\
			\text{s.t.} & (x_m,y_m,z_m) \in X_m, \\
			& z_m = x_n^k.
			\end{array}
			\end{align*}
			\item[(b)] keep 
			$\overline{\mathfrak{Q}}_m^{k+1}(\cdot) = \overline{\mathfrak{Q}}_m^{k}(\cdot)$
			for other nodes $n' \in \mathcal{N}_{t(n)}$ such that $n' \neq n$;
			\item[(c)] set $n = a(n)$
		\end{enumerate}
		
		\State Increase $k$ by $1$ and go to step $1$.
	\end{algorithmic}
\end{algorithm}

\begin{lemma} \label{lemma:finite_policy}
With probability one,
the sequence of expected cost-to-go
approximations~$\{\overline{\mathfrak{Q}}_n^k\}_{k\in\mathbb{N}}$
generated by Algorithm~\ref{alg:sample_ald} converges to a Lipschitz function~$\overline{\mathfrak{U}}_n$ with constant~$\widetilde{L}_n$ 
after a finite number of iterations.
Moreover, the following relationships hold for every node $n$ of the tree:
\begin{gather}
\label{eq:induction_sample_hypo}
\lim\limits_{k\in\mathcal{K}_n} \overline{\mathfrak{Q}}_n^k(x_n^k) = \overline{\mathfrak{U}}_n(x_n^*), \\ \label{eq:induction_sample_ineq}
0 \leq \overline{Q}_n(x_n^*) - \overline{\mathfrak{U}}_n(x_n^*)
  \leq (\widetilde{L} + \rho) \cdot \delta \cdot (T - t(n)),
\end{gather}
where~$\mathcal{K}_n$ is a subset of indices from~$\mathcal{J}_n$ such that
the sequence of policy states~$\{x_n^k\}_{k\in\mathcal{K}_n}$ converges,
$x_n^*$ is the corresponding limit point,
$\widetilde{L}$ is the maximum Lipschitz constant~$\widetilde{L}_n$ 
and $\rho$ is the maximum penalty constant~$\rho_n$ over all nodes~$n\in\mathcal{N}$.
\end{lemma}
\begin{proof}
	We start proving the finite convergence of $\Qfrak_n^k$ to $\Ufrak_n$,
	and the limit in~\eqref{eq:induction_sample_hypo} by backward induction on the tree.
	In the last stage this result is trivial since both 
	functions~$\overline{\mathfrak{Q}}_n^k$ and~$\overline{Q}_n$ are
	identically zero.

	Let~$n$ be a node such that the statement~\eqref{eq:induction_sample_hypo} holds for
	every successor node $m \in \mathcal{S}(n)$.
	Recall that the updating rule of the reverse norm
	cut has the form:
	\begin{equation*}
	\overline{\mathfrak{Q}}_n^{k+1}(x_{n}) = \max\left\{\overline{\mathfrak{Q}}_n^{k}(x_{n}),\ \overline{Q}_n^{k+1}(x_{n}^k) - \rho_n \cdot \|x_{n} - x_{n}^k\| \right\},
	\end{equation*}
	where the expected cost-to-go approximation~$\overline{Q}_n^{k+1}$
	is the weighted average of the cost-to-go approximations~$Q_m^k$
	over the successor nodes~$m\in\mathcal{S}(n)$.
	By the induction hypothesis,
	after a finite number of iterations we obtain
	\begin{equation}\label{eq:opt_state}
	\begin{array}{rl}
	Q_m^k(x_{n}) = \underset{x_m,y_m,z_m}{\min} & f_m(x_m,y_m) + \overline{\mathfrak{U}}_m(x_m) \\
	\text{s.t.} & (x_m,y_m,z_m) \in X_m, \\
	& z_n = x_{n}.
	\end{array}
	\end{equation}	
	In other words, both the cost-to-go~$Q_m^k$ and 
	the expected cost-to-go~$\overline{Q}_n^k$ approximations 
	stabilize after a finite number
	of iterations, so denote by~$U_m$ and~$\overline{U}_n$ the 
	corresponding limits, respectively.
	Since the number of different incoming states~$x_n^k$ is also finite,
	this implies that the number of different possible reverse norm 
	cuts to update~$\overline{\mathfrak{Q}}_n^k$ is also finite.
	Then,~$\overline{\mathfrak{Q}}_n^k$ converges to a
	function~$\overline{\mathfrak{U}}_n$ in a finite number of iterations.

	Now, let's prove inequality~\eqref{eq:induction_sample_ineq}
	by backward induction on the tree.
	It is trivial at the leaf nodes, so suppose
	inequality~\eqref{eq:induction_sample_ineq} holds for
	every successor node~$m \in \mathcal{S}(n)$.
	Since there is a finite number of different
	possible policy states at the node~$n$,
	the sequence~$\{x_n^k\}_{k\in\mathcal{K}_n}$ converges to~$x_n^*$ in a finite number of iterations,
	which means that the reverse norm 
	cut~$\overline{U}_n(x_n^*) - \rho_n \cdot \|x_n - x_n^*\|$
	is also considered in the expected cost-to-go 
	limit~$\overline{\mathfrak{U}}_n$.
	In particular, we have the following inequalities:
	\[
	\overline{U}_n(x_n^*) - \rho_n \cdot \|x_n - x_n^*\| 
	\leq \overline{\mathfrak{U}}_n(x_n) \leq
	\overline{U}_n(x_n),
	\]
	where the last inequality results from 
	Lemma~\ref{lemma:ctg_desig}.
	But the expected cost-to-go approximations~$\overline{\mathfrak{U}}_n$
	and~$\overline{U}_n$ are equal at~$x_n^*$, 
	so we obtain the following equation for the
	difference between~$\overline{\mathfrak{U}}_n$
	and the true cost-to-go function~$\overline{Q}_n$:	
	\begin{align} \nonumber
	\overline{Q}_n(x_{n}^*) - \overline{\mathfrak{U}}_n(x_{n}^*) 
	& = \overline{Q}_n(x_{n}^*) - \overline{U}_n(x_n^*) \\ 
	\label{eq:sampled_equality}
	& = \sum_{m \in \mathcal{S}(n)} q_{nm}\cdot \left[Q_m(x_n^*) - U_m(x_n^*)\right].
	\end{align}	
	Now, we have the crucial part of the argument.
	Because the expected cost-to-go 
	approximations of all nodes stabilize, every incoming state
	of any successor node~$m\in\mathcal{S}(n)$ 
	equals~$x_n^*$ after a large number of iterations.
	Then, the optimal solution of node~$m$ with input 
	state~$x_n^*$ is equal to~$u_m$, which is less 
	than~$\delta$ away from the final state~$x_m^*$ of node~$m$,
	by the design of the Sampled-SLDP algorithm.
	Then, we obtain the following inequalities:
	\begin{align}\nonumber
	Q_m(x_{n}^*) - U_m(x_{n}^*) & \leq 
	\overline{Q}_m(u_m) - \overline{\mathfrak{U}}_m(u_m) \\ \nonumber
	& \leq \widetilde{L} \cdot \| u_m - x_m^* \| + \overline{Q}_m(x_m^*) - \overline{\mathfrak{U}}_m(u_m) \\ \nonumber
	& \leq \overline{Q}_m(x_m^*) - \overline{\mathfrak{U}}_m(x_m^*) + (\widetilde{L}+\rho)\cdot \|u_m - x_m^*\|
	\end{align}
	where the first inequality results from~$u_m$ being 
	the optimal policy's state of node~$m$ with input
	state~$x_n^*$,
	and the following ones
	from the Lipschitz property of~$\overline{Q}_m$
	and~$\overline{\mathfrak{U}}_m$, respectively.
	By our induction hypothesis,
	\[
		\Qbar_m(x_m^*) - \Ufrak_m(x_m^*) \leq (\widetilde{L} + \rho) \cdot \delta \cdot (T-t(m)),
	\]
	and since $t(m) = t(n) + 1$ we get
	\begin{equation}
		\label{eq:sampled_bound}
		Q_m(x_{n}^*) - U_m(x_{n}^*)
		\leq (\widetilde{L}+\rho) \cdot \delta \cdot (T-t(n)).
	\end{equation}
	because~$u_m$ and~$x_m^*$ are at most~$\delta$ far away from each other.
	So, the upper bound~\eqref{eq:sampled_bound} together with equation~\eqref{eq:sampled_equality} concludes the
	induction step.
\end{proof}

\begin{thm}
	With probability~$1$,
	the sequence of lower bounds~$\{Q_1^k\}_{k\in\mathbb{N}}$
	generated by Algorithm~\ref{alg:sample_ald}
	converges in a finite number of iterations
	to an $\eps$-approximation of the true optimal value $Q_1$,
	where $\eps = (\widetilde{L} + \rho)\cdot\delta\cdot(T-1)$,
	and every cluster point of the sequence of feasible 
	policies~$\{(x^k,y^k,z^k)\}_{k\in\mathbb{N}}$ generated by the forward
	step of Algorithm~\ref{alg:sample_ald} is an $\eps$-optimal policy.
\end{thm}
\begin{proof}
	This is a straightforward result of Lemma~\ref{lemma:finite_policy},
	using the same reasoning as in Theorem~\ref{thm:cvg_sldp_ft}.
\end{proof}

\section{Examples}\label{sec:examples}

In this section, we will present two applications of the SLDP algorithm
for stochastic optimization.
The first is a simple example of a 1-dimensional dynamics with discrete control.
Due to its symmetry and relative simplicity,
it is possible to evaluate the cost-to-go functions,
so that we can understand the behavior of the algorithm in its different forms.
The second one has been extracted from~\cite{CaroeSchultz97} and~\cite{Ahmed2SSIP},
and is a 2-stage problem, for which enumeration can be performed
in order to also evaluate the optimal solution and cost-to-go function.

\subsection{Implementation details}

The non-convex cuts used in SLDP are represented as inequalities of the form
\[
  \alpha \geq v + \lambda^\top(x - x^k) - \rho \cdot \|x - x^k\|,
\]
where $\lambda = 0$ for the reverse-norm cuts,
but is needed for the augmented Lagrangian cuts.
To incorporate them in the stage problems,
this requires choosing a norm,
and a MIP formulation of this constraint.
For the experiments below, we have used the $L^1$ norm
\begin{equation}
  \label{eq:1norm}
  \|x - x^k\|_1 = \sum_j |x_j - x_j^k|,
\end{equation}
and each term in~\eqref{eq:1norm} is given by the sum $(u_j^+ + u_j^-)$
from the following system:
\begin{equation} \label{eq:mip_form_abs}
\begin{aligned}
  u_j^+ - u_j^- & = x_j - x_j^k \\
  0 \leq u_j^+ & \leq M_j z_j \\
  0 \leq u_j^- & \leq M_j (1 - z_j) \\
  z_j & \in \{0,1\} \\
\end{aligned}
\end{equation}
The constants $M_j$ are large enough so that $\proj_{x_j}(X)$
has diameter less than $M_j$,
which is ensured by the compactness assumption of $\proj_x(X)$.

Observe that this formulation includes a binary variable (and two continuous variables),
per dimension of $x$, for each new non-convex cut we introduce.
This makes each iteration of the SLDP algorithm much more expensive than previous ones.

\medskip

One practical implementation of the SLDP method uses augmented Lagrangian cuts,
and increases $\rho$ progressively.
Since by construction the augmented Lagrangian cuts are valid,
if $\rho$ is not large enough then the cuts might not be tight,
but they might fill faster the non-convex regions of the cost-to-go function.
Also, in analogy to Strengthened Benders cuts,
it is possible to fix both the Lagrange multiplier and the augmenting term,
and solve the resulting augmented Lagrangian relaxation.
This again yields a valid cut,
which we call \emph{strengthened augmented Benders cut}.

\medskip

All results below were obtained using
Julia-0.6.3~\cite{julia}
and the Julia packages \texttt{SDDP.jl}~\cite{SDDP.jl}
and \texttt{SDDiP.jl}~\cite{SDDiP.jl},
besides our own Julia implementation for both
Lipschitz and strengthened augmented Benders cuts~\cite{sldp.jl},
extending \texttt{SDDP.jl}.
The computations were performed on an
Intel(R) Xeon(R) CPU E5-2603 CPU.

\subsection{A 1-dimensional control problem}

We consider the following multistage control problem:
\[
\begin{array}{rl}
  \min  & \mathbb{E}\left[\sum\limits_{t=1}^T \beta^{t-1} \abs{x_t} \right] \\
\text{s.t.} & \quad x_t = x_{t-1} + c_t + \xi_t \\
            & \quad c_t \in \{\pm 1\}
\end{array}
\]
The state variable $x_t$ is 1-dimensional, as the discrete control $c_t = \pm 1$,
and the uncertainty $\xi_t$.
The objective is to minimize the expected displacement away from zero,
subject to a decay factor $\beta$,
over the planning horizon $T$.
We fix $T = 8$, $\beta = 0.9$, $x_0 = 2$,
and at each stage $t$ we consider 10 independent scenarios symmetrically sampled around 0.


We will compare the performance and the policy generated by several methods:
a convex approximation using Strengthened Benders cuts (shortened as SB),
the original SLDP algorithm using reverse-norm cuts (SLDP tents),
a modified SLDP algorithm using ALD cuts with increasing augmentation $\rho$ (SLDP ALD),
and the SDDiP algorithm~\cite{SDDiP2016},
using two discretization steps: $0.1$ and~$0.01$.
The resulting discretized problems for SDDiP don't have complete continuous recourse,
since the state cannot absove the noise below the discretization level,
and we only have a discrete control,
so we also add a slack variable and penalize it in the objective function.
The original problem, with continuous state,
doesn't need adjustments.

We present in Table~\ref{tab:control}
the lower bounds,
the estimated upper bounds using policy simulations
and the computation times
after 100 iterations for each method.
\begin{table}[hb]
  \begin{tabular}{r|lllll}
             & SB    & SLDP tents & SLDP ALD & SDDiP 0.1 & SDDiP 0.01 \\ \hline
    LB       & 1.167 &   3.073    & 3.085    &  3.420    & 2.370 \\
    UB       & 3.453 &   3.320    & 3.313    &  3.823    & 3.490 \\
    time (s) & 12    &   558      & 605      &  1994     & 3317  \\
  \end{tabular}
  \caption{Results for an 8-stage non-convex problem}
  \label{tab:control}
\end{table}

The convex approximations stall at a very low lower bound,
while the non-convex methods all have better estimates,
but they also need significantly more computation time.
The SLDP approximations have a very similar performance ---
the ALD method requiring slightly more time.
SDDiP has a relatively good performance with step $0.1$,
but not with $0.01$.
Observe that the higher lower bound for SDDiP with step $0.1$
also comes with a higher upper bound,
which is due to the addition of the penalization term and a loose state discretization.
When the discretization step is the smaller $0.01$,
the upper bounds of the simulation agree more closely with the other cases,
but we spend $66\%$ more in computation time,
and the lower bounds are much further away.

As we can see in figure~\ref{fig:Qbar_control},
the future cost functions are nonconvex at all time stages,
essentially driven by the discontinuous control $c_t$:
the immediate cost is $\min\{\abs{u-1},\abs{u+1}\}$,
where $u = x_{t-1} + \xi_t$.
\begin{figure}[!hbtp]
  \centering
  \includegraphics[width=0.9\textwidth]{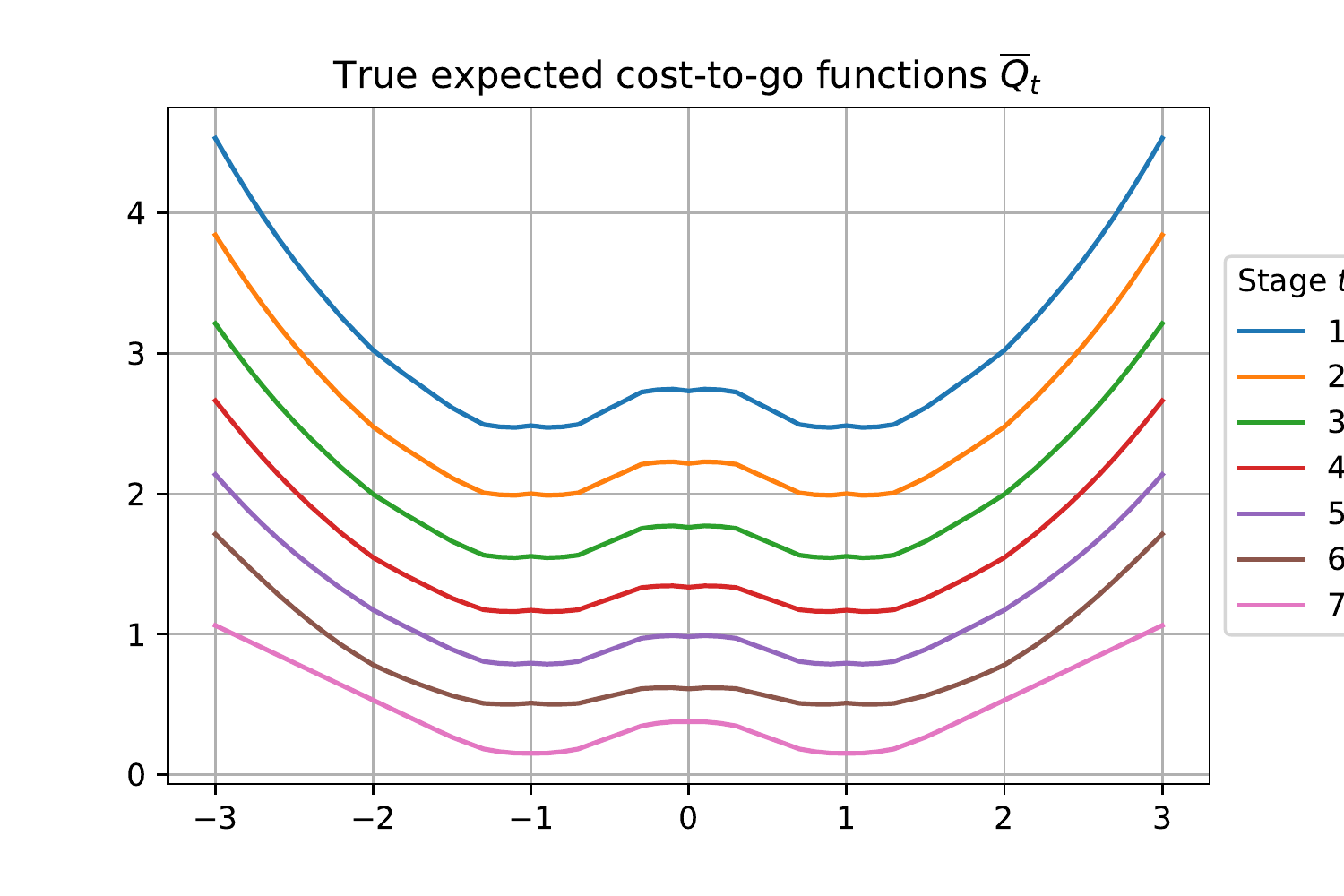}
  \caption{Expected cost-to-go functions}
  \label{fig:Qbar_control}
\end{figure}
However, the future cost functions built by the convex Strengthened Benders cuts
can't pierce into the nonconvexities, and become flat over $[-1,1]$,
as depicted in figure~\ref{fig:QLP_control}.
This explains why the convex approximations perform so poorly in this case.
\begin{figure}[!hbtp]
  \centering
  \includegraphics[width=0.9\textwidth]{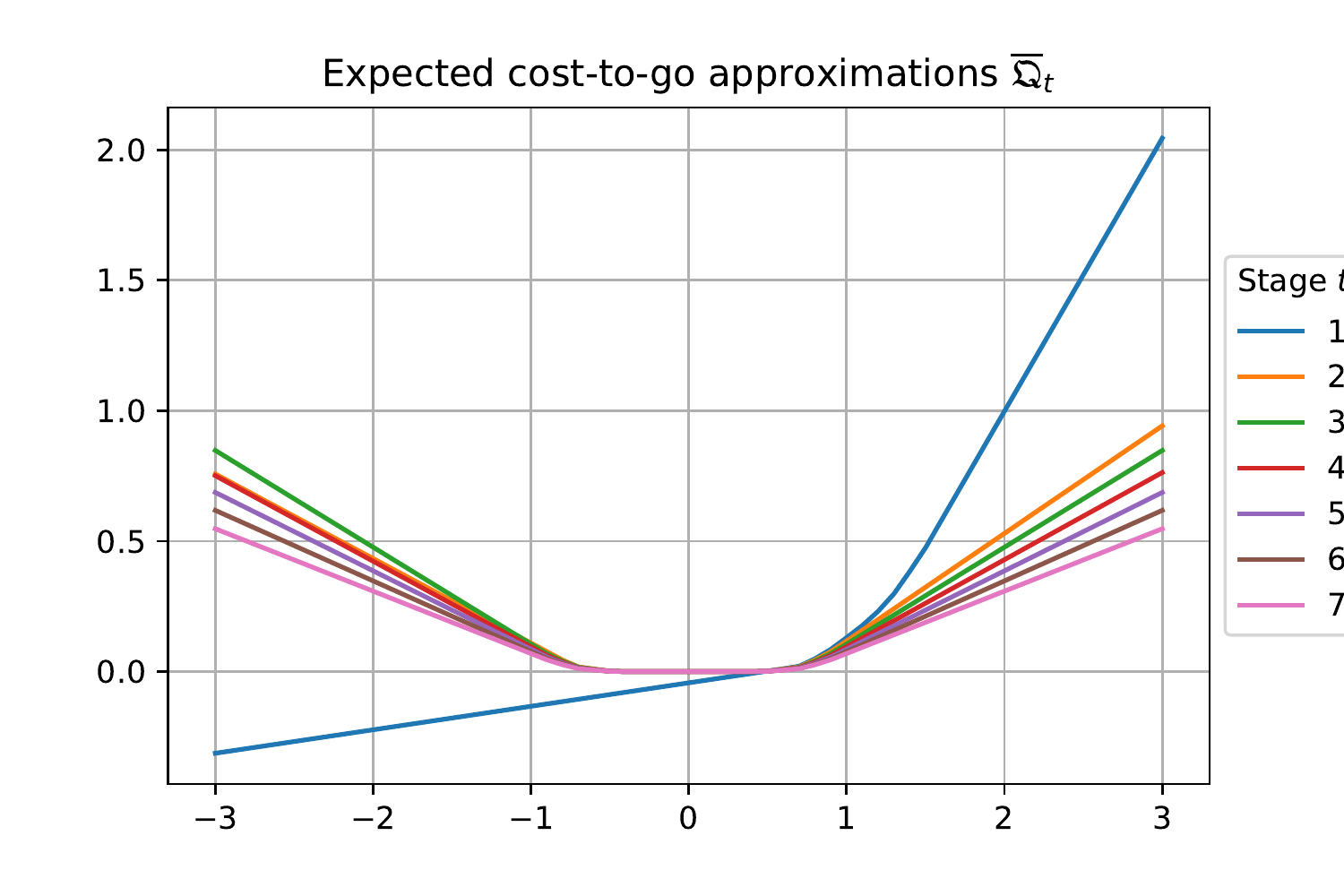}
  \caption{Convex approximations of the expected cost-to-go functions, built with Strengthened Benders cuts}
  \label{fig:QLP_control}
\end{figure}

The Lipschitz cuts, on the other hand,
are indeed able to yield a better approximation of the problem,
and approximate the expected cost-to-go functions inside their nonconvexities.
The same happens with the value function obtained by SDDiP.
In figure~\ref{fig:Qtilde_control},
we show a comparison of the expected cost-to-go functions thus constructed,
using reverse-norm cuts, augmented Lagrangian cuts, SDDiP,
and the actual future cost function.
\begin{figure}[!hbtp]
  \centering
  \includegraphics[width=\textwidth]{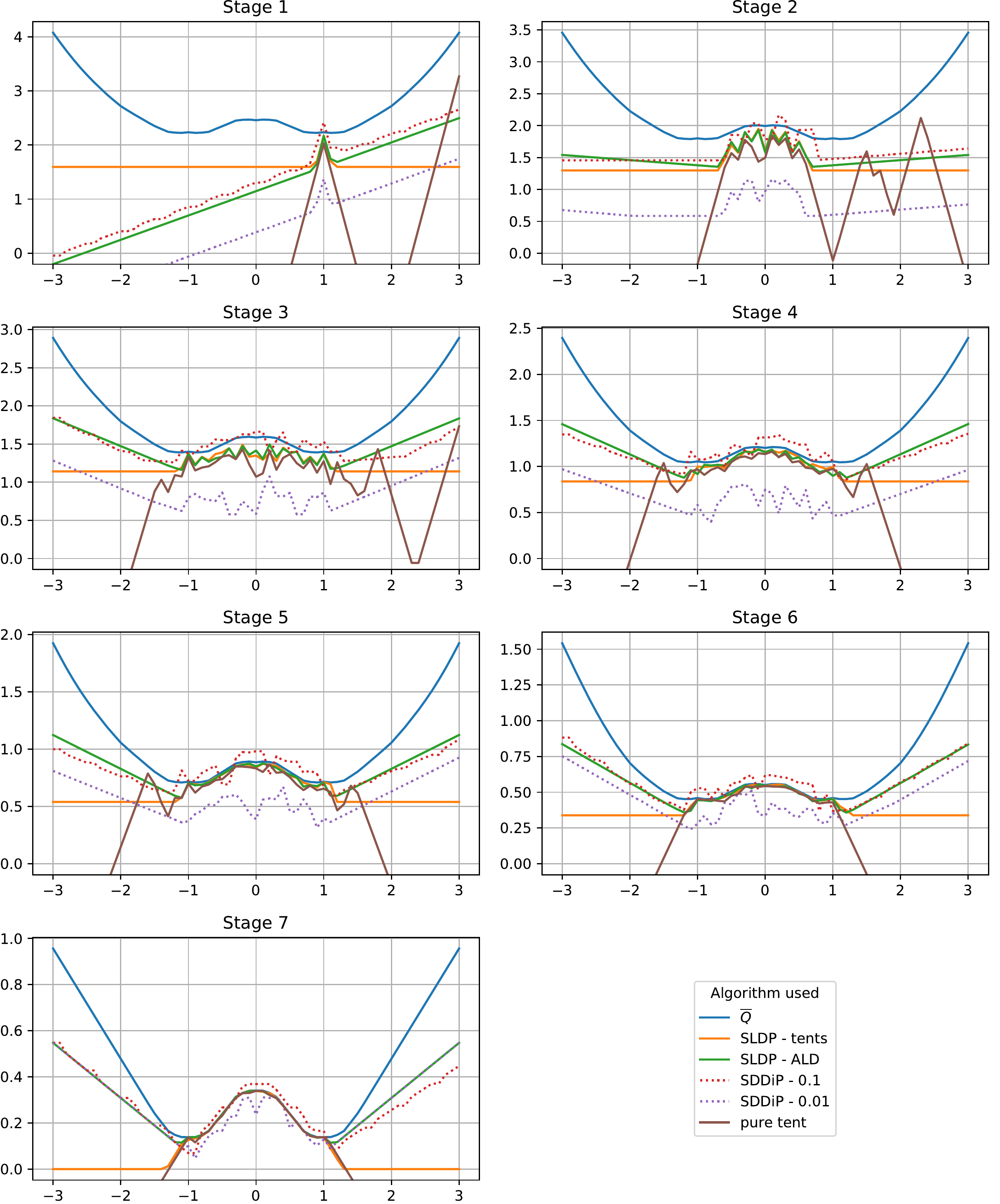}
  \caption{Non-convex approximations of the expected cost-to-go functions}
  \label{fig:Qtilde_control}
\end{figure}

One particular feature of this example is that the future cost functions are continuous
in the original variables $x_t$,
but since the control is $c_t \in \{-1,1\}$,
the stage problems lack the continuous recourse property.
For this reason, when one takes the SDDiP discretization of the state variable,
one must also include a slack term to the state dynamics.
This increases the costs overall, and explains why the value functions estimated with
the ``coarse'' discretization with $\eps = 0.1$ are higher than
the true expected cost-to-go functions $\Qbar_t$.
The lower bounds $\Qfrak_t$ obtained with the ``fine'' discretization with $\eps = 0.01$,
on the other hand, ``detach'' much faster from the respective $\Qbar_t$
as we move back in the stages of the tree.

Finally, we compare in figure~\ref{fig:Timing_control}
the evolution of the lower bounds, both as iterations and time increase.
\begin{figure}[!hbtp]
  \centering
  \includegraphics[width=\textwidth]{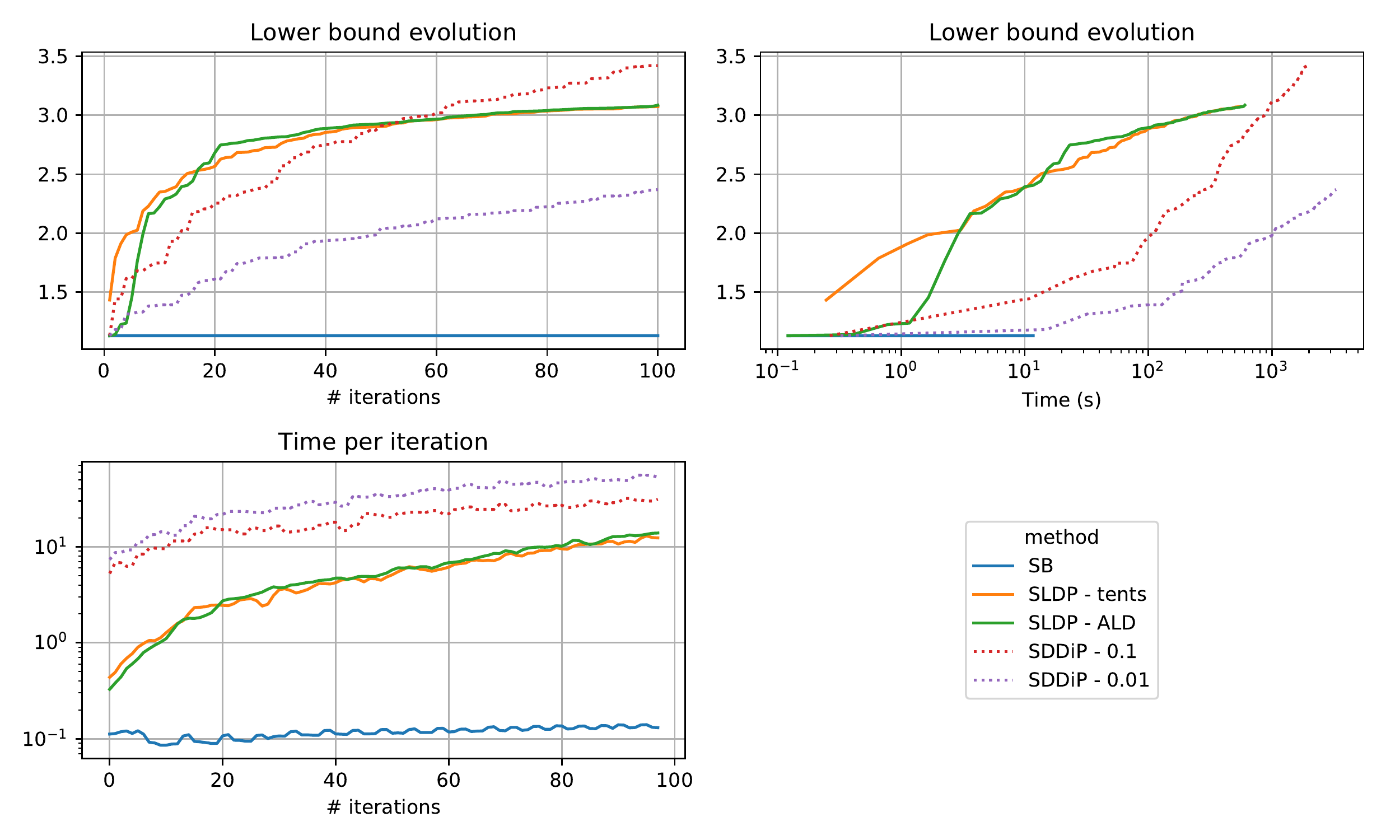}
  \caption{Several timing comparisons: lower bound evolution and time per iteration}
  \label{fig:Timing_control}
\end{figure}
There, we see that both methods for SLDP were essentially equivalent,
maybe except at the beginning, where the ALD's $\rho$ was probably
too small to yield good enough cuts.
However, as iterations progressed, and $\rho$ was large enough,
the algorithm quickly reached a comparable lower bound.
We also include a curve for the time taken per iteration,
to highlight the rapid increase in time for the SLDP algorithm,
even in a 1-dimensional problem.
This same phenomenon happens for SDDiP, on a much smaller scale,
but it already starts out with a significantly larger iteration time.

\subsection{A 2-dimensional example}

We also study another example,
taken from~\cite{CaroeSchultz97} and further adapted in~\cite{Ahmed2SSIP}.
This is a 2-stage problem, with cost-to-go function
\[
  Q(x_1,x_2,\omega_1,\omega_2) =
  \begin{array}[t]{rl}
    \min & -16y_1 -19y_2 -23y_3 -28y_4 \\
    \text{s.t.} & 2y_1 + 3y_2 + 4y_3 + 5y_4 \leq \omega_1 - x_1 \\
                & 6y_1 +  y_2 + 3y_3 + 2y_4 \leq \omega_2 - x_2 \\
                & y_i \in \{0,1\}.
\end{array}
\]
The random variable $\omega$ lies in $[5,15]^2$,
and is approximated using $N^2$ points on the square,
where $N = 2,3,6$.

In~\cite{CaroeSchultz97}, the authors consider the optimization problem
with discrete first-stage decisions:
\begin{equation} \label{caroeschultz}
\begin{array}{rl}
  \min  & -1.5 x_1 - 4 x_2 + \mathbb{E}[Q(x_1,x_2,\omega_1,\omega_2)] \\
  \text{s.t.} & x_1, x_2 \in [0,5] \cap \mathbb{Z}.
\end{array}
\end{equation}
The same value function for the second stage
can be used for a \emph{continuous} first-stage problem,
as treated by~\cite{Ahmed2SSIP},
where one drops the constraint that the first-stage variables belong to $\mathbb{Z}$.

In this case, the cost-to-go function is discontinuous in $x$,
so for both the SLDP and SDDiP we would need to add slack variables
and their corresponding penalization in the objective function.
Therefore, for this experiment we used only the convex approximations
and the ALD cuts, which are valid despite the discontinuities of the value function.
The convex approximations were calculated using 100 iterations
(but stalled much before that),
while the ALD approximations were carried for 200 iterations.
This is more than the number of possible values for the state variable $x$,
in the discrete case.

We give in tables~\ref{tab:caroe_int} and~\ref{tab:caroe_cont},
respectively for the discrete and continuous first stage variables,
the lower bounds and computation time required for each method.
We also include the true objective for each value of $N$.
In all cases, the convex cuts using Strenghtened Benders
are not able to close the gap,
so we report the remaining gap for the SLDP-ALD as the ratio
\[
  \frac{\text{Objective} - \text{ALD LB}}{\text{Objective} - \text{SB LB}}
\]
of the remaining gap.
In the discrete case of table~\ref{tab:caroe_int},
we obtain exact results for $N = 2$ and $N = 3$,
and a significant reduction for the case $N = 6$.
The fact that it does not converge is due to the need of having even larger values for $\rho$
to get tight cuts for every node in the second stage.
The continuous case is harder, and we have gaps now at $N = 3$ as well.
Besides the difficulties of achieving tight cuts as in the discrete case,
the algorithm also needs to explore several points in the neighborhood of the optimal solution.

It is remarkable that the times required by the Lipschitz cuts is much smaller
in the discrete setting than in the continuous setting,
\emph{contrary} to what happens for the convex case,
which solves a harder problem in the first stage and therefore takes slightly more time.
This is probably explained by the SLDP algorithm constructing cuts at the same points,
and therefore the resulting stage problems don't become much more difficult as times passes,
as opposed to the continuous case, where the nodes for each cut are probably different.
Also note that all those times, and especially the ALD times,
are much larger than the ones reported in~\cite{Ahmed2SSIP}.
Besides a different computational setting,
we don't explore the fact that the technology matrix is deterministic.

{
\def\z{\hphantom{0}}
\def\m{\hphantom{-}}
\def\c#1{\multicolumn{1}{c}{#1}}
\begin{table}[hb]
  \centering
  \begin{tabular}{r|lll}
          N        &  \c2    &  \c3    &    \c6    \\ \hline
    Objective      & -57.000 & -59.333 &   -61.222 \\ \hline
    SB LB          & -58.096 & -61.961 &   -65.468 \\
    ALD LB         & -57.000 & -59.333 &   -61.274 \\
    remaining (\%) & \m\z  0 & \m\z  0 & \m\z 1.27 \\ \hline
    SB time  (s)   & \z 0.98 & \z 1.60 & \z\z 5.57 \\
    ALD time (s)   &   19.4  &   77.8  &    125.0  \\
    ALD/SB time    &   19.8  &   48.6  &  \z 22.4  \\
  \end{tabular}
  \caption{Results for discrete first stage}
  \label{tab:caroe_int}
\end{table}

\begin{table}[hb]
  \centering
  \begin{tabular}{r|lll}
          N        &   \c2     &    \c3    &    \c6    \\ \hline
    Objective      &  -57.000  &   -59.333 &   -61.222 \\ \hline
    SB LB          &  -58.095  &   -61.961 &   -65.541 \\
    ALD LB         &  -57.000  &   -59.495 &   -61.579 \\
    remaining (\%) & \m\z   0  & \m\z 6.27 & \m\z 8.31 \\ \hline
    SB time (s)    & \z\z 0.63 & \z\z 1.30 & \z\z 5.22 \\
    ALD time (s)   &    260    &    532    &    758    \\
    ALD/SB time    &    413    &    409    &    145    \\
  \end{tabular}
  \caption{Results for continuous first stage}
  \label{tab:caroe_cont}
\end{table}
}

\FloatBarrier


\section{Conclusion}

In this paper, we proposed a new algorithm for solving
stochastic multistage MILP programming problems,
called Stochastic Lipschitz Dynamic Programming (SLDP).
Its major contribution is the inclusion of \emph{nonlinear cuts}
to iteratively underapproximate
\emph{nonconvex} Lipschitz cost-to-go functions.
We explored two such families of cuts:
(a) the ones induced by reverse-norm penalizations;
and (b) augmented Lagrangian cuts,
built from norm-augmented Lagrangian duality.

Assuming the Compact State Complete Continuous Recourse conditions,
we proved convergence of the algorithm in the full scenario setting.
In the sampled case, we provided an approximation method to reach
$\eps$-optimal policies in finite time.
Besides asymptotic convergence in the general Stochastic Multistage Lipschitz case,
it would be interesting to prove a finite convergence result for
Stochastic MILPs in the sampled case by further exploring the structure
of the value functions in each stage.

Our experiments suggest that, at least for small-dimensional problems,
the performance of the SLDP algorithm is reasonable.

\section*{Acknowledgements}

This research has been supported in part by
the National Science Foundation grant 1633196,
the Office of Naval Research grant N00014-18-1-2075,
the COPPETec project IM-21780. 

The second author would like to express his gratitude to the Brazilian
Independent System Operator (ONS) for its support for this research.

This research project was concluded while the third author
was visiting Georgia Tech during a sabbatical leave from UFRJ.
He would like to warmly thank the hospitality and the
excelent environment of the ISyE institute.

\bibliographystyle{apalike}
\bibliography{refs.bib}

\end{document}